\documentclass[a4paper]{amsart}
\usepackage{amsmath,amsthm,amssymb, amscd, amsfonts}
\usepackage[all]{xy}
\usepackage{leftidx}
\usepackage[latin1]{inputenc}
\usepackage{enumerate}
\usepackage{color}
\usepackage{multirow,array}
\usepackage{float}

\theoremstyle{plain}

\newtheorem{theorem}{Theorem}[section]
\newtheorem{corollary}[theorem]{Corollary}

\newtheorem{proposition}[theorem]{Proposition}
\newtheorem{lemma}[theorem]{Lemma}

\theoremstyle{definition}
\newtheorem{definition}[theorem]{Definition}

\newtheorem{example}[theorem]{Example}

\theoremstyle{remark}
\newtheorem{remark}[theorem]{Remark}

\numberwithin{equation}{section}\theoremstyle{plain}

\newcommand{\uno}{\textbf{1}}

\newcommand{\A}{{\mathcal A}}
\newcommand{\B}{{\mathcal B}}
\newcommand{\C}{{\mathcal C}}
\newcommand{\D}{{\mathcal D}}
\newcommand{\Ll}{{\mathcal L}}

\newcommand{\Z}{{\mathcal Z}}
\newcommand{\Zz}{{\mathbb Z}}
\newcommand{\M}{\mathcal{M}}

\newcommand{\Ss}{\mathbb{S}}

\newcommand{\E}{{\mathcal E}}

\newcommand{\Aa}{\mathbb A}

\newcommand{\Rep}{\operatorname{Rep}}
\newcommand{\Sg}{\operatorname{Sg}}
\newcommand{\cd}{\mathrm{cd}}

\newcommand\Aut{\operatorname{Aut}}
\newcommand\Irr{\operatorname{Irr}}
\newcommand\Ind{\operatorname{Ind}}
\newcommand\FPdim{\operatorname{FPdim}}

\newcommand\vect{\operatorname{Vect}}
\newcommand\svect{\operatorname{sVect}}
\newcommand\id{\operatorname{id}}

\newcommand\Tr{\operatorname{Tr}}

\newcommand\co{\operatorname{co}}

\newcommand\Hom{\operatorname{Hom}}

\begin{document}
\title[Fusion rules and solvability]{On fusion rules and solvability of a fusion
category}
\author{Melisa Esca\~ nuela Gonz\' alez and Sonia Natale}
\address{Facultad de Matem\'atica, Astronom\'\i a y F\'\i sica.
Universidad Nacional de C\'ordoba. CIEM -- CONICET. Ciudad
Universitaria. (5000) C\'ordoba, Argentina}
\email{escanuela@famaf.unc.edu.ar}
\email{natale@famaf.unc.edu.ar
\newline \indent \emph{URL:}\/ http://www.famaf.unc.edu.ar/$\sim$natale}

\thanks{This work was partially supported by  CONICET and SeCYT--UNC}



\date{December 17, 2015}

\begin{abstract} We address the question whether the condition on a fusion category being solvable or not is determined by its fusion rules. We prove that the answer is affirmative for some families of non-solvable examples arising from representations of semisimple Hopf algebras associated to exact factorizations of the symmetric and alternating groups. In the context of spherical fusion categories, we  also consider the invariant provided by the $S$-matrix of the Drinfeld center and show that this invariant does determine the solvability of a fusion category provided it is group-theoretical.
   \end{abstract}

\maketitle

\section{Introduction}

Throughout this paper we shall work over an algebraically closed field $k$ of characteristic zero. Let $G$ be a finite group. An important invariant of $G$ is given by its \emph{character table}, defined as the collection $$\{\chi_i(g_j)\}_{0\leq i, j \leq n},$$ where $\epsilon = \chi_0, \dots, \chi_n$ are the irreducible characters of $G$ over $k$ and $e = g_0, \dots, g_n$, are representatives of the conjugacy classes of $G$. Several structural properties of $G$ can be read off from its character table. For instance, the character table of $G$ allows to determine the lattice of normal subgroups of $G$ and to decide if the group $G$ is nilpotent or solvable. See \cite[pp. 23]{isaacs}. It is known, however, that the character table of a finite solvable group $G$ does not determine its derived length \cite{mattarei}, \cite{mattarei-2}.

In particular, if $G$ and $\Gamma$ are finite groups with the same character table, then $G$ is solvable if and only if $\Gamma$ is solvable.   
In addition, the knowledge of the character table of a finite group $G$ is equivalent to the knowledge of the structure constants, in the canonical basis consisting of isomorphism classes of irreducible representations, of the Grothendieck ring of the fusion category $\Rep G$ of finite dimensional representations of $G$ over $k$, so-called the \emph{fusion rules} of $\Rep G$. 

\medbreak The notions of nilpotency and solvability of a group $G$ have been extended to general fusion categories in \cite{gel-nik}, \cite{ENO2}. Let $\C$ be a fusion category over $k$. Then $\C$ is \emph{nilpotent} if there exists a series of fusion subcategories
\begin{equation}\label{cs-nil}\vect = \C_0 \subseteq \C_{1} \subseteq \dots \subseteq \C_n = \C,\end{equation}
and a series of finite groups $G_1, \dots, G_n$, such that $\C_i$ is a $G_i$-extension of $\C_{i-1}$, for all $i = 1, \dots, n$.
On the other side, $\C$ is \emph{solvable} if  there exist a sequence of fusion categories $\vect = \C_0,
\dots, \C_n = \C$, $n \geq 0$, and a sequence of cyclic groups of prime order
$G_1, \dots, G_n$, such that, for all $1\leq i \leq n$, $\C_i$ is a
$G_i$-equivariantization or a $G_i$-extension of $\C_{i-1}$. See Subsection \ref{nilp-solv}. Some features related to nilpotency and solvability have been extended as well from the context of finite groups to that of fusion categories; remarkably, an analogue of Burnside's $p^aq^b$-theorem was established for fusion categories in \cite{ENO2}.

\medbreak It is apparent from the definition of nilpotency of a fusion category $\C$ given in \cite{gel-nik} that this property depens only upon the Grothendieck ring of $\C$, that is, it is determined by its fusion rules.
In this paper we address the question whether the solvability of a fusion category $\C$ is also determined by its fusion rules.

Since a solvable fusion category has nontrivial invertible objects and a simple group has no nontrivial one-dimensional representation, then no solvable fusion category can have the same fusion rules as a simple finite group.
We show that if $\C$ is a fusion category with the same fusion rules as a dihedral group, then $\C$ is solvable. On the other hand, if $\C$ has the fusion rules of a symmetric group $\Ss_n$, $n \geq 5$, then $\C$ is not solvable; Theorem \ref{dihedral-fr} and Corollary  \ref{fr-sn}.

\medbreak We study some families of examples of non-solvable fusion categories arising from representations of semisimple Hopf algebras associated to exact factorizations of the symmetric group $\Ss_n$ and the alternating group $\Aa_n$.  For a wide class of such fusion categories $\C$, we show that $\C$ cannot have the fusion rules of any solvable fusion category. See Theorems \ref{jp-kp}, \ref{dual-jn-kn}, \ref{bn*} and \ref{bn}.

\medbreak In the context of braided fusion categories, the solvability of a fusion category $\C$ is related to the existence of Tannakian subcategories of $\C$; it is known that if $\C$  is a non-pointed integral solvable braided fusion category, then it must contain a nontrivial Tannakian subcategory \cite[Lemma 5.1]{witt-wgt}.

We show that if $\tilde \C$ is a non-pointed braided fusion category which has the same  fusion rules as a solvable fusion category $\C$, then $\tilde \C$ contains a nontrivial Tannakian subcategory. See Theorem \ref{e-tann}.

\medbreak For a spherical fusion category $\C$ we study a somehow stronger invariant, analogous to the character table of a finite group, consisting of the $S$-matrix of the Drinfeld center $\Z(\C)$ of $\C$.  Indeed, the $S$-matrix of a modular category $\D$ is usually named the 'character table' of $\D$ in the literature; see for instance \cite{ganchev}.  A celebrated formula due to Verlinde, and valid for any modular category, implies that the $S$-matrix of $\Z(\C)$ determines its fusion rules. We call two spherical fusion categories $S$-equivalent if their Drinfeld centers have 'the same' $S$-matrix; see Subsection \ref{s-equiv}.

We prove in Theorem \ref{s-equiv-gt} that the $S$-matrix of the Drinfeld center does determine the solvability of  a group-theoretical fusion category. That is, if $\C$ and $\D$ are $S$-equivalent spherical fusion categories and $\C$ is group-theoretical, then $\C$ is solvable if and only if $\D$ is solvable.
We also show  that being group-theoretical is a property invariant under $S$-equivalence, that is, it is a property determined by the $S$-matrix of the Drinfeld center; see Theorem \ref{s-gpttic}.

\medbreak The paper is organized as follows. Section \ref{preli} contains the main notions and facts on fusion categories that will be needed in the rest of the paper. In Section \ref{g-crossed} we study the notion of Grothendieck equivalence of fusion categories and its connection with solvability, and prove some results on the fusion rules of dihedral and symmetric groups. 
In Section \ref{nsol-fr} we consider examples of non-solvable fusion categories arising from exact factorizations of the symmetric and alternating groups. The case of braided fusion categories is studied in Section \ref{solv-fr-bfc}.  Finally, in Section \ref{s-char-tbl} we study the notion of $S$-equivalence of spherical fusion categories. 

\section{Preliminaries}\label{preli}

The category of finite dimensional vector spaces
over $k$ will be denoted by $\vect$. A fusion category over $k$
is a semisimple rigid monoidal category over $k$ with finitely many isomorphism
classes of simple objects, finite-dimensional Hom spaces, and such that the unit
object $\uno$ is simple.
Unless otherwise stated, all tensor categories will be assumed to be strict.
We refer the reader to \cite{ENO}, \cite{DGNOI}, for the main notions on fusion  categories used throughout.

\subsection{Fusion categories} Let $\C$ be a fusion category over $k$. The Grothendieck group $K_0(\C)$ is a free abelian group with basis $\Irr(\C)$ consisting of isomorphism classes of simple objects of $\C$. For an object $X$ of $\C$, let us denote by $[X]$ its class in $K_0(\C)$.

The tensor product of $\C$ endows $K_0(\C)$ with a ring structure with unit element $[\uno]$ and such that, for all objects $X$ and $Y$ of $\C$, $[X][Y] = [X\otimes Y]$. Let $X, Y \in \Irr(\C)$. Then one can write
$$X Y = \sum_{Z \in \Irr(\C)}N^Z_{X, Y} \, Z,$$
where $N^Z_{X, Y}$ are non-negative integers, for all $X, Y, Z \in \Irr(\C)$. The collection of numbers $\{N^Z_{X, Y}\}_{X, Y, Z}$ are called the \emph{fusion rules} of $\C$ and they determine the ring structure of $K_0(\C)$. They are given by the formula $$N^Z_{X, Y} = \dim \Hom_\C(Z, X \otimes Y),$$ for all $X, Y, Z \in \Irr(\C)$.
In the terminology of \cite[Subsection 2.1]{gel-nik}, the pair $(K_0(\C),\Irr(\C))$ is a unital \emph{based} ring.

A \emph{fusion subcategory} of $\C$ is  a full tensor subcategory $\D$ such that $\D$ is replete and  stable under direct summands. Fusion subcategories of $\C$ are in bijective correspondence with subrings of $K_0(\C)$ spanned by a subset of $\Irr(\C)$, that is, based subrings of $K_0(\C)$.

\medbreak The Frobenius-Perron dimension of a
simple object $X \in \C$ is, by definition, the Frobenius-Perron eigenvalue of
the matrix of left multiplication by the class of $X$ in the basis $\Irr(\C)$ of
the Grothendieck ring of $\C$ consisting of isomorphism classes of simple
objects. The Frobenius-Perron dimension of $\C$ is the number $\FPdim \C =
\sum_{X \in \Irr(\C)} (\FPdim X)^2$.

We shall indicate by $\cd(\C)$ the set of Frobenius-Perron dimensions of simple objects of $\C$. If $1 = d_0, d_1, \dots, d_r$ are distinct positive real numbers and $n_1, \dots, n_r$ are natural numbers, we shall say that $\C$ is of \emph{type} $(d_0, n_0; d_1, n_1; \dots; d_r, n_r)$ if $\C$ has $n_i$ isomorphism classes of simple objects of Frobenius-Perron dimension $d_i$, for all $i = 0, \dots, r$.

The group of invertible objects of $\C$ will be denoted by $G(\C)$. Thus $G(\C)$ coincides with the subset of elements $Y$ of $\Irr(\C)$ such that $\FPdim Y = 1$. Thus, if $\C$ is of type $(1, n_0; d_1, n_1; \dots; d_r, n_r)$, then  $n_0 = |G(\C)|$.

\medbreak The category $\C$ is called \emph{integral}  if
$\FPdim X \in \mathbb Z$, for all simple object $X \in \C$, and it is called
\emph{weakly integral} if $\FPdim \C \in \mathbb Z$.

\medbreak
Recall that a right module category over a fusion category $\C$ is a finite semisimple $k$-linear abelian category $\M$ endowed with a bifunctor $\otimes: \M \times \C \to \M$ satisfying the associativity and unit axioms for an action, up to coherent natural isomorphisms.  The module category $\M$ is called indecomposable
if it is not equivalent as a module category to a direct sum of non-trivial module categories. If $\M$ is an indecomposable module category over $\C$, then the category $\C^*_{\mathcal M}$ of $\C$-module endofunctors of $\M$ is also a fusion category.

Two fusion categories $\C$ and $\D$
are \emph{Morita equivalent} if $\D$ is equivalent to
$\C^*_{\mathcal M}$ for some indecomposable module
category $\mathcal M$. If $\C$ and $\D$ are Morita equivalent fusion categories, then $\FPdim \C = \FPdim \D$. 

By \cite[Theorem 3.1]{ENO2}, the fusion categories $\C$ and $\D$ are Morita equivalent if and only if its Drinfeld centers are equivalent as braided fusion categories.

\medbreak A fusion category $\C$ is \emph{pointed} if all its simple objects are invertible. If $\C$ is a pointed fusion category, then there exist a finite group $G$ and a 3-cocycle $\omega$ on $G$ such that $\C$ is equivalent to the category $\C(G, \omega)$ of finite-dimensional $G$-graded vector spaces with associativity constraint defined by $\omega$.  
A fusion category Morita equivalent to a pointed fusion category is called \emph{group-theoretical}. 

\subsection{Nilpotent and solvable fusion categories}\label{nilp-solv} Let $G$ be
a finite group. A $G$-grading on a fusion category $\C$ is a decomposition $\C =
\oplus_{g\in G} \C_g$, such that $\C_g \otimes \C_h \subseteq \C_{gh}$ and
$\C_g^* \subseteq \C_{g^{-1}}$, for all $g, h \in G$. A $G$-grading is \emph{faithful} if $\C_g \neq 0$, for all $g \in G$.
The fusion category $\C$ is called a \emph{$G$-extension} of a
fusion category $\D$ if there is a faithful grading $\C = \oplus_{g\in G} \C_g$
with neutral component $\C_e \cong \D$.

If $\C$ is any fusion category, there exist a finite group $U(\C)$, called the
\emph{universal grading group} of $\C$, and a canonical faithful grading $\C =
\oplus_{g \in U(\C)}\C_g$, with neutral component $\C_e = \C_{ad}$, where
$\C_{ad}$ is the \emph{adjoint} subcategory of $\C$, that is, the fusion subcategory generated by $X\otimes
X^*$, $X \in \Irr(\C)$.

In fact, $K_0(\C)_{ad} = K_0(\C_{ad})$ is a based subring of $K_0(\C)$ and $K_0(\C)$ decomposes into a direct sum of indecomposable based $K_0(\C)_{ad}$-bimodules $K_0(\C)=\displaystyle{\oplus_{g\in U(\C)}}K_0(\C)_g$, with $K_0(\C)_e=K_0(\C)_{ad}$. Then the group structure on $U(\C) := U(K_0(\C))$ is defined by the following property: $gh = t$ if and only if
 $X_g X_h\in K_0(\C)_t$, for all $X_g\in K_0(\C)_g$, $X_h\in K_0(\C)_h$, $g,h,t\in U(\C)$; see  \cite[Theorem 3.5]{gel-nik}.

\medbreak
A fusion category $\C$ is (cyclically) \emph{nilpotent} if there exists a
sequence of fusion categories $\vect = \C_0 \subseteq \C_1 \dots \subseteq \C_n
= \C$, and finite
(cyclic) groups $G_1, \dots, G_n$, such that for all $i = 1, \dots, n$, $\C_i$
is a $G_i$-extension of $\C_{i-1}$.

On the other side, $\C$ is \emph{solvable} if it is Morita equivalent to a cyclically nilpotent fusion category, that is, if there exists a cyclically nilpotent fusion category $\D$ and an idecomposable right module category $\M$ over $\D$ such that $\C$ is equivalent to the fusion category $\D^*_\M$ of $\D$-linear endofunctors of $\M$.

\medbreak Consider  an action of a finite group $G$ on a fusion
category $\C$ by tensor autoequivalences $\rho: \underline
G \to \underline \Aut_{\otimes} \, \C$.
The \emph{equivariantization} of $\C$ with respect to the action $\rho$, denoted
$\C^G$, is a fusion category whose objects are pairs  $(X, \mu)$, such that $X$
is an object of $\C$ and $\mu = (\mu^g)_{g \in G}$, is a collection of
isomorphisms $\mu^g:\rho^gX \to X$, $g \in G$, satisfying appropriate
compatibility conditions.

The forgetful functor $F: \C^G \to \C$, $F(X, \mu) = X$,
is a dominant tensor functor that gives rise to a central exact sequence of
fusion categories $\Rep G \to \C^G \to \C$ \cite{indp-exact}, where $\Rep G$ is
the category of finite-dimensional representations of $G$.

The category $\C^G$ is integral (respectively, weakly integral) if and only if so is $\C$. See \cite[Proposition 4.9]{tensor-exact}, \cite[Proposition 2.12]{indp-exact}.

\medbreak According to \cite[Definition 1.2]{ENO2}, a fusion category $\C$ is solvable if and only if there exists a sequence of fusion categories $\vect = \C_0,
\dots, \C_n = \C$, $n \geq 0$, and a sequence of cyclic groups of prime order
$G_1, \dots, G_n$, such that, for all $1\leq i \leq n$, $\C_i$ is a
$G_i$-equivariantization or a $G_i$-extension of $\C_{i-1}$.

It is shown in \cite[Proposition 4.1]{ENO2} that the class of  solvable fusion
categories is stable under taking extensions
and equivariantizations
by solvable groups, Morita equivalent categories, tensor products, Drinfeld
center, fusion subcategories and components of quotient categories.

In view of \cite[Proposition 4.5 (iv)]{ENO2}, every nontrivial solvable fusion category has nontrivial invertible objects.

\medbreak Suppose that the finite group $G$ acts on the fusion category $\C$ by tensor autoequivalences. Let $Y \in \Irr \C$. The stabilizer of $Y$ is the subgroup  $G_Y = \{g\in G:\,  \rho^g(Y) \cong  Y \}$. Let $\alpha_Y : G_Y \times G_Y \to k^*$ be the 2-cocycle defined by the relation
\begin{equation}\label{alfa} \alpha_Y(g, h)^{-1} \id_Y = c^g \rho^g(c^h)({\rho^{g, h}_{2_Y}})^{-1}(c^{gh})^{-1}: Y \to Y, \end{equation}
where, for all $g \in G_Y$, $c^g: \rho^g(Y) \to Y$ is a fixed isomorphism \cite[Subsection 2.3]{fusionrules-equiv}.

\medbreak Then the simple objects of $\C^G$ are parameterized by pairs $(Y, U)$, where $Y$ runs over the $G$-orbits on $\Irr(\C)$ and $U$ is an equivalence class of an  irreducible $\alpha_Y$-projective representation of $G_Y$. We shall use the notation $S_{Y, U}$ to indicate the isomorphism class of the simple object corresponding to the pair $(Y, U)$. The dimension of $S_{Y, U}$ is given by the formula
\begin{equation}\label{dim-equiv}\FPdim S_{Y, U} =  [G:G_Y] \dim U \FPdim Y.\end{equation}

\begin{lemma}\label{simple-p} Let $p$ be a prime number. Suppose that the group $\Zz_p$ acts on a fusion category $\C$ by tensor autoequivalences. Assume in addition that $G(\C^{\Zz_p})$ is of order $p$ and $G(\C) \neq \{\uno\}$. Then $\C^{\Zz_p}$ has a simple object of Frobenius-Perron dimension $p$.
\end{lemma}

\begin{proof} Let $Y$ be an invertible object of $\C$ and let $U$ be an irreducible $\alpha_Y$-projective representation of the subgroup $G_Y \subseteq \Zz_p$. Since $G_Y$ is cyclic, then $\alpha_Y = 1$ in $H^2(G_Y, k^*)$ and $\dim U = 1$.  Then the Frobenius-Perron dimension of the simple object $S_{Y, U}$ is given by $\FPdim S_{Y, U} =  [G:G_Y] \FPdim Y = [G:G_Y]$. Moreover, if $Y = \uno$, then $G_Y = \Zz_p$. Therefore, letting $U_0 = \epsilon, U_1, \dots, U_{p-1}$ the non-isomorphic representations of $\Zz_p$, we get that $\uno = S_{\uno, U_0}, S_{\uno, U_1}, \dots, S_{\uno, U_{p-1}}$ are  all the non-isomorphic invertible objects of $\C^{\Zz_p}$. Hence for all invertible object $Y \neq \uno$ of $\C$, we must have $[G: G_Y] = p$ and the simple object $S_{Y, U}$ has Frobenius-Perron dimension $p$. This proves the lemma.
\end{proof}

\begin{proposition}\label{cyc-nilp} Let $p$ be a prime number. Suppose that $\C$ is a solvable fusion category such that $G(\C) \cong \Zz_p$ and $\C$ has no simple objects of Frobenius-Perron dimension $p$. Then $\C$ is cyclically nilpotent.
\end{proposition}

\begin{proof} The proof is by induction on $\FPdim \C \geq p$. If $\FPdim \C = p$ there is nothing to prove. Suppose $\FPdim \C > p$. Since $\C$ is solvable then, for some prime number $q$, $\C$ must be a $\Zz_q$-extension or a $\Zz_q$-equivariantization of a fusion category $\D$. If the second possibility holds, then the assumption that $G(\C) \cong \Zz_p$ implies that $q = p$. Moreover, since $\D$ is also solvable, then $\D_{pt} \neq \vect$. By Lemma \ref{simple-p}, $\C$ must have a simple object of dimension $p$, which contradicts the assumption.

Therefore $\C$ must be a $\Zz_q$-extension of a fusion subcategory $\D$. In particular, $\D$ cannot have simple objects of dimension $p$ and since $\D$ is solvable, then $\D_{pt} \neq \vect$, whence $G(\D) = G(\C) \cong \Zz_p$. By induction, $\D$ and then also $\C$, is cyclically nilpotent. This finishes the proof of the proposition.
\end{proof}

\begin{lemma}\label{g-central} Let $\C$ be a fusion category and let $G$ be a
finite group acting
on $\C$ by tensor autoequivalences.
Then the forgetful functor $U: \Z(\C^G) \to \C^G$ induces an injective ring
homomorphism $K_0(G) \to Z(K_0(\C^G))$. In
particular,
the group $\widehat G$ is isomorphic to a subgroup of the center of $G(\C^G)$.
\end{lemma}

\begin{proof}
By \cite[Proposition 2.10]{ENO2}, the Drinfeld center $\Z(\C)$
contains a Tannakian subcategory $\E \cong \Rep G$ such that $\E$ embeds into
$\C$ under the forgetful functor $U: \Z(\C) \to \C$.
As a consequence we obtain the lemma.
\end{proof}

\subsection{Braided fusion categories} A braided fusion category is a fusion category $\C$ endowed with a braiding, that is, a
 natural isomorphism $c_{X,Y} : X \otimes Y \rightarrow Y \otimes X$, $X, Y \in \C$, subject to the
so-called hexagon axioms.

\medbreak If $\D$ is a fusion subcategory of a braided fusion category $\C$,
the \emph{M\" uger centralizer} of $\D$ in $\C$ will be denoted by $\D'$. Thus $\D'$ is the full fusion
subcategory generated by all objects $X \in \C$ such that $c_{Y, X}c_{X, Y} =
\id_{X \otimes Y}$, for all objects $Y \in \D$.

\medbreak The centralizer $\C'$ of $\C$ is called the \emph{M\" uger (or symmetric) center} of $\C$.
The category $\C$ is called \emph{symmetric} if $\C' = \C$. If $\C$ is any
braided fusion category, its M\" uger center $\C'$ is a symmetric fusion
subcategory of $\C$. The category  $\C$ is  called
\emph{non-degenerate} (respectively, \emph{slightly degenerate})
if $\C' \cong \vect$ (respectively, if $\C' \cong \svect$, where $\svect$ denotes the category of super-vector spaces).

\medbreak For a fusion category $\C$, the Drinfeld center of
$\C$ will be
denoted $\Z(\C)$. It is known that $\Z(\C)$ is a braided non-degenerate fusion
category
of Frobenius-Perron dimension $\FPdim \Z(\C) = (\FPdim \C)^2$.

\medbreak
Let $G$ be a finite group. The fusion category $\Rep G$ of
finite dimensional representations of $G$ is a symmetric fusion category with respect to the canonical braiding.
A braided fusion category $\E$ is called Tannakian, if $\E \cong \Rep G$ for
some finite group $G$ as braided fusion categories. 

Every symmetric fusion category is equivalent, as a braided fusion category, to the category $\Rep(G, u)$ of representations of a finite group $G$ on finite-dimensional super-vector spaces, where $u \in G$ is a central element of order 2 which acts as the parity operator \cite{deligne}. In particular, if $\C$ is symmetric, then it is equivalent to the category of representations of a finite group as a fusion category.

\medbreak  Let $G$ be a finite group. A \emph{$G$-crossed braided fusion
category} is a fusion category $\D$ endowed with a $G$-grading $\D
= \oplus_{g \in G}\D_g$ and an action of $G$ by tensor autoequivalences
$\rho:\underline G \to \underline \Aut_{\otimes} \, \D$, such that $\rho^g(\D_h)
\subseteq
\D_{ghg^{-1}}$, for all $g, h \in G$, and a $G$-braiding $c: X \otimes Y \to
\rho^g(Y) \otimes X$, $g \in G$, $X \in \D_g$, $Y \in \D$, subject to
compatibility conditions. The $G$-braiding $c$ restricts to a braiding in the neutral component $\D_e$.

If $\D$ is a $G$-crossed braided fusion category, then
the equivariantization $\D^G$ under the action of $G$ is a braided fusion category containing $\Rep G$ as a Tannakian subcategory. Furthermore, the group $G$ acts by restriction on
$\D_e$ by braided tensor autoequivalences. The equivariantization $\D_e^G$ coincides with the centralizer $\E'$ of the
Tannakian subcategory $\E$ in $\D^G$. See  \cite{mueger-crossed}.

\medbreak Let $\E$ be  Tannakian subcategory of a braided fusion category $\C$ and let $G$ be a finite group such that $\E
\cong \Rep G$ as symmetric categories. Let also $A \in \C$ be the algebra corresponding to the algebra $k^G \in \Rep G$ of functions on $G$ with the regular action.
The de-equivariantization $\C_G$ of
$\C$ with respect to $\Rep G$ is the fusion category $\C_A$ of right $A$-modules in $\C$. This is a $G$-crossed braided fusion category such that $\C \cong (\C_G)^G$. The  neutral component of $\C_G$ with respect to the associated $G$-grading, denoted by $\C_G^0$, coincides with the de-equivariantization ${\E'}_G$ of the centralizer of $\E$ by the group $G$.

It was shown in \cite[Proposition ]{witt-wgt} that if  $\E \cong \Rep G
\subseteq  \C$ is a Tannakian subcategory, then $\C$ is
weakly integral
(respectively, integral or weakly group-theoretical) if and only if
$\C^0_G$ is weakly integral  (respectively,
integral, weakly group-theoretical). In addition,  $\C$ is solvable if and only
if $\C^0_G$ is solvable and $G$ is solvable.

\begin{lemma} Let $\C$ be a braided fusion category. Then the subcategory
$\C_{ad} \cap \C_{pt}$ is symmetric.
\end{lemma}

\begin{proof} Suppose first that $\C$ is non-degenerate. Then $\C_{ad} =
\C_{pt}'$, by \cite[Corollary 3.27]{DGNOI}. Therefore $\C_{ad} \cap \C_{pt} =
\C_{pt}' \cap \C_{pt}$ is a symmetric subcategory.

Next, for an arbitrary braided fusion category $\C$, let $\Z(\C)$ be the
Drinfeld center of $\C$. Since $\Z(\C)$ is non-degenerate, then the category
$\Z(\C)_{ad} \cap \Z(\C)_{pt}$ is symmetric. The braiding of $\C$ induces a
canonical embedding of braided fusion categories $\C \to \Z(\C)$. We may
therefore identify $\C$ with a fusion subcategory of $\Z(\C)$. Observe that
$\C_{ad} \subseteq \Z(\C)_{ad}$ and  $\C_{pt} \subseteq \Z(\C)_{pt}$. Hence
$\C_{ad} \cap \C_{pt} \subseteq  \Z(\C)_{ad} \cap \Z(\C)_{pt}$, and then
$\C_{ad} \cap \C_{pt}$ is symmetric, as claimed.
\end{proof}

\begin{lemma}\label{cent-cpt} Let $\C$ be a braided fusion category such that
$\C_{ad} = \C$. Then  $\C_{pt}' = \C$.
\end{lemma}

\begin{proof}Let $\B \subseteq \C$ be any fusion subcategory. By
\cite[Proposition 3.25]{DGNOI} we have $(\B_{ad})' = (\B')^{co} = \A$, where $\A
\subseteq \C$ denotes the projective centralizer of $\B$.
Letting $\B = \C_{pt}$ we find that $\C = (\B_{ad})'$ equals the
projective centralizer of $\C_{pt}$.
By  \cite[Lemma 3.15]{DGNOI}, the projective centralizer of a fusion subcategory
$\B$ is a graded
extension of the centralizer $\B'$. Since $\C = \C_{ad}$, this implies that  $\C
= \C_{pt}'$, as claimed.
\end{proof}

\section{Grothendieck equivalence of fusion categories}\label{g-crossed}

Let $\C$ and $\tilde\C$ be fusion categories. A \textit{Grothendieck equivalence} between $\C$ and $\tilde\C$ is a bijection $f:\Irr \C \rightarrow \Irr \tilde\C$ such that
\begin{equation}
f(\uno)=\uno, \quad \text{and } N_{f(X),f(Y)}^{f(Z)}=N_{X,Y}^Z,
\end{equation} for all $X, Y, Z \in \Irr \C$.

We shall say that $\C$ and $\tilde\C$ are \textit{Grothendieck equivalent} if there exist a Grothendieck equivalence between them.

\begin{remark}\label{biy} Suppose $f: \Irr \C \rightarrow \Irr \tilde\C$ is a Grothendieck equivalence. Then the map $f$ extends canonically to a ring isomorphism $f: K_0(\C)\rightarrow K_0(\tilde \C)$.

In particular,  $f$ induces a bijection between the lattices of fusion subcategories of $\C$ and $\tilde \C$.
If $\D$ is a fusion subcategory of $\C$, we shall denote by $f(\D)$ the corresponding fusion subcategory of $\tilde\C$, that is,
$f(\D)$ is the fusion subcategory whose simple objects are $f(X)$, $X \in \Irr \D$.
Note that $f$ restricts to a Grothendieck equivalence $f: \Irr\D \to \Irr f(\D)$.
\end{remark}

\begin{proposition}\label{groth-eq}
Let $\C$ and $\tilde\C$ be fusion categories and suppose that $\mathit{f:\Irr\C\rightarrow \Irr\tilde\C}$ is a \textit{Grothendieck equivalence}. Then the  following hold:

(i)   If $X\in K_0(\C)$, then \emph{FPdim}$(f(X))$ = \emph{FPdim}$(X)$. Hence, if $\D$ is a fusion subcategory of $\C$, then \emph{FPdim}$(f(\D))$= \emph{FPdim}$(\D)$.

(ii) $X\in \Irr \C$ is invertible if and only if $f(X)\in \Irr \tilde \C$ is invertible.

(iii) If $X \in \Irr \C$, then $f(X^*) = f(X)^*$.

(iv) $f(\C^{(n)})=\tilde \C^{(n)}$, for all $n \geq 0$. In particular,  $\C$ is nilpotent if and only if $\tilde \C$ is nilpotent.

(v) $f$ induces a group isomorphism $f: U(\C) \to U(\tilde \C)$ such that $f(\C_g) = \tilde \C_{f(g)}$.
\end{proposition}

\begin{proof}
(i)  By  Remark \ref{biy} we know that $f$ extends to a ring isomorphism
$ f:K_0(\C)\rightarrow K_0(\tilde \C)$. By \cite[Lemma 8.3]{ENO} $\text{FPdim}:K_0(\C)\rightarrow \mathbb{R}$ is the only ring homomorphism such that $\FPdim (X)>0$ for any $0\neq X\in\C$, so $\FPdim (f(X)) = \FPdim (X)$, for all $X \in \Irr \C$.

(ii) This follows from  (i), since the invertible objects  of a fusion category are exactly those objects with Frobenius-Perron dimension 1.

(iii) Since $f(\uno) = \uno$, then $N_{f(X), f(X^*)}^\uno = N_{X, X^*}^\uno = 1$. Therefore $f(X^*) = f(X)^*$.

(iv) It follows from (iii) and the fact that $f$ preserves fusion rules that $f(\C_{ad})=\tilde \C_{ad}$. Then $f$ induces by restriction a Grothendieck equivalence $\Irr \C_{ad} \to \Irr \tilde \C_{ad}$. An inductive argument implies that $f(\C^{(n)})=\tilde \C^{(n)}$, for all $n \geq 0$.

(v)  By definition $U(\C)=U(K_0(\C))$ \cite{gel-nik} and $K_0(\C)$ decomposes into a direct sum of indecomposable based $K_0(\C)_{ad}$-bimodules $K_0(\C)=\displaystyle{\oplus_{g\in U(\C)}}K_0(\C)_g$, with $K_0(\C)_e=K_0(\C)_{ad}$. This decomposition is unique up to a permutation of $U(\C)$.
By Remark \ref{biy}, $f$ extends to a ring isomorphism $f:K_0(\C)\rightarrow K_0(\tilde \C)$ and by (iv) $f$ restricts to a ring isomorphism $K_0(\C)_{ad} = K_0(\C_{ad}) \cong K_0(\tilde \C_{ad})= K_0(\tilde \C)_{ad}$. So for all $g\in U(\C)$, $f(K_0(\C)_g) = K_0(\tilde\C)_{\tilde g}$, for a unique $\tilde g\in U(\tilde\C)$. Letting $f(g) = \tilde g$, we obtain a group isomorphism $f: U(\C) \to U(\tilde \C)$ such that $f(K_0(\C_g)) = K_0(\tilde \C_{f(g)})$. This implies (v).
\end{proof}

\begin{remark}\label{grading} Let $G$ be a finite group. Observe that any $G$-grading on a fusion category $\C$ with neutral component
$\D$ is uniquely determined by a $G$-grading on the Grothendieck ring $K_0(\C)$
with neutral component $K_0(\D)$. In particular, if $\C$ and $\tilde\C$ are
Grothendieck equivalent, then $\C$ is $G$-graded with neutral component $\D$ if
and only if
$\tilde \C$ is $G$-graded with neutral component $\tilde \D$, such that $\tilde
\D$ and $\D$ are Grothendieck equivalent.
\end{remark} 

Our first theorem concerns fusion categories with dihedral fusion rules.

\begin{theorem}\label{dihedral-fr} Let $n$ be a natural number and let $\C$ be a fusion category. Suppose that $\C$ is Grothendieck equivalent to the category $\Rep D_n$, where $D_n$ is the dihedral group of order $2n$. Then $\C$ is solvable.
\end{theorem}

\begin{proof} It follows from \cite[Theorem 4.2]{naidu-rowell} that a fusion category Grothendieck equivalent to the representation category of a dihedral group is group-theoretical. Then $\C$ is group-theoretical, that is, it is Morita equivalent to a pointed fusion category $\C(\Gamma, \omega)$, where $\Gamma$ is a group and $\omega$ is a 3-cocycle on $\Gamma$. 

Suppose first that $n$ is odd. Then the order of $\Gamma$ is equal to $2n$ and, since $n$ is odd, $\Gamma$ is solvable. Then $\C$ is solvable too.

If $n$ is even, then the center of $D_n$ is of order $2$ and $D_n/Z(D_n) \cong D_{n/2}$. Therefore, the category $\Rep D_n$ is a $\Zz_2$-extension of $\Rep D_{n/2}$; see \cite[Example 3.2]{gel-nik}. Since $\C$ is Grothendieck equivalent to $\Rep D_n$, then it is a $\Zz_2$-extension of a fusion subcategory $\D_1$, where $\D_1$ is Grothendieck equivalent to $\Rep D_{n/2}$. 
Continuing this process, we find that the category $\C$ is obtained by a sequence of $\Zz_2$-extensions from a fusion subcategory $\D$ such that $\D$ is Grothendieck equivalent to $\Rep D_{m}$, with $m$ an odd natural number. 
By the above, $\D$ is solvable and therefore so is $\C$. This finishes the proof of the theorem. \end{proof}

The following consequence of Proposition \ref{cyc-nilp} gives some restrictions that guarantee that the solvability of a fusion category is a Grothendieck invariant.

\begin{proposition}\label{cor-cyc} Let $p$ be a prime number. Suppose that $\C$ is a solvable fusion category such that $G(\C) \cong \Zz_p$ and $\C$ has no simple objects of Frobenius-Perron dimension $p$. If $\C$ is Grothendieck equivalent to a fusion category $\tilde \C$, then $\tilde \C$ is solvable.
\end{proposition}

\begin{proof} By Proposition \ref{cyc-nilp}, $\C$ is cyclically nilpotent. Therefore  $\tilde \C$ is cyclically nilpotent, whence solvable.
\end{proof}

\begin{remark}\label{rmk-ansn-2} For all $n \geq 2$, the alternating group $\Aa_{n}$ has no irredubible
representation of degree $2$\footnote{This can be seen, for instance, as a consequence of the
Nichols-Richmond theorem \cite[Theorem 11]{NR}}.
In addition, if $n \geq 5$ ($\Rep \Ss_{4}$ is of type $(1, 2; 2, 1; 3, 2)$),  the symmetric group $\Ss_{n}$ has no irredubible representation of degree $2$ neither. In fact, if $V$ were such a representation, then the restriction $V\vert_{\Aa_{n}}$ would not be irreducible. Hence, since  $\Aa_{n}$ has no nontrivial one-dimensional representations (because $n \geq 5$), then $V\vert_{\Aa_{n}}$ would be trivial. This is impossible, because the kernel of the restriction functor $\Rep \Ss_{n} \to \Rep \Aa_{n}$ is the pointed subcategory $\Rep \Zz_2 \subseteq \Rep \Ss_{n}$.
\end{remark}

\begin{corollary}\label{fr-sn} Let $n\geq 5$ be a natural number and let $\C$ be a fusion category. Suppose that $\C$ is Grothendieck equivalent to $\Rep \Ss_n$. Then $\C$ is not solvable.
\end{corollary}

\begin{proof} The category $\Rep \Ss_n$ is not solvable. On the other hand, the group $\Ss_n$ has two non-equivalent representations of degree one and no irreducible representation of degree two, in view of Remark \ref{rmk-ansn-2}. Hence $G(\C) \cong \Zz_2$ and $\C$ has no simple objects of Frobenius-Perron dimension $2$. The result is thus obtained as a consequence of Proposition \ref{cor-cyc}.
\end{proof}

\begin{remark}\label{fr-simple} Let $G$ be a non-abelian finite simple group. If $\C$ is a fusion category Grothendieck equivalent to $\Rep G$, then $\C_{pt} = \vect$ and therefore $\C$ is not solvable.   
\end{remark}

\section{Examples of non-solvable fusion rules}\label{nsol-fr}

\subsection{Abelian extensions}
Consider an abelian exact sequence of Hopf algebras
\begin{equation}\label{abel-es}
k \longrightarrow k^\Gamma \overset{i}\longrightarrow H
\overset{\pi}\longrightarrow kF
\longrightarrow k,
\end{equation} where $\Gamma$ and $F$ are finite groups.
Then \eqref{abel-es} gives rise to actions by permutations
$\Gamma \overset{\lhd}\longleftarrow
\Gamma \times F \overset{\rhd}\longrightarrow F$ such that $(\Gamma, F)$ is a
matched pair of groups.
Moreover, $H \cong k^\Gamma {}^\tau\#_\sigma kF$ is a bicrossed product with
respect to normalized invertible
$2$-cocycles $\sigma: F \times F \to k^\Gamma$, $\tau: \Gamma \times \Gamma \to
k^F$, satisfying suitable compatibility conditions. See \cite{ma-ext}.

The multiplication and comultiplication of $k^\Gamma {}^\tau\#_\sigma kF$ are determined in the basis $\{e_{s}\#x /s\in\Gamma, x\in F\}$, by the formulas
\begin{align}
(e_s \# x)\otimes(e_t \# y) & = \delta_{t,s\lhd x} \, \sigma_s(x, y) \, e_s \# xy, \\
\Delta(e_s \# x) & = \sum_{gh=s} \tau_x(g, h) \, e_g \# (h \rhd x) \otimes e_h \#x,
\end{align}
for all $s, t \in \Gamma$, $x, y \in F$, where $\sigma_s(x,y)=\sigma(x,y)(s)$ and $\tau_x(s, t)=\tau(s,t)(x)$.  See \cite{ma-ext}.
The exact sequence \eqref{abel-es} is called \emph{split} if $\sigma$ and $\tau$ are the trivial
2-cocycles.

\medbreak For all $s\in\Gamma$, the restriction of the map $\sigma_s:F \times F \to k^{\times}$ to the stabilizer subgroup $F_s  = F \cap sFs^{-1}$ is a 2-cocycle on $F_s$.

The irreducible representations of $H \cong k^\Gamma {}^\tau\#_\sigma kF$ are classified for pairs $(s,U_s)$, where $s$ is a representative of the orbits of the action of $F$ in $\Gamma$ and $U_s$ is an irreducible representation of the  twisted group algebra $k_{\sigma_s}F_s$, that is, a projective irreducible representation $F_s$ with cocycle $\sigma_s$. Given a pair $(s,U_s)$, the corresponding irreducible representation is given by
\begin{equation}\label{irrepns-bcpt}W_{(s,U_s)} = \Ind_{k^\Gamma \otimes kF_s}^H s \otimes U_s.\end{equation}
Observe that $\dim W_{(s,U_s)} = [F: F_s] \dim U_s$. See \cite{MW}.

\begin{remark}\label{2-morita} Recall that every matched pair $(\Gamma, F)$ gives rise to a group structure, denoted $F \bowtie \Gamma$, on the product $F \times \Gamma$ in the form $$(x, s) (y, t) = (x(s \rhd y), (s\lhd y) t),$$ $x, y \in F$, $s, t \in \Gamma$, where $\Gamma \overset{\lhd}\longleftarrow
\Gamma \times F \overset{\rhd}\longrightarrow F$ are the associated compatible actions.

The group $F \bowtie \Gamma$ has a canonical exact factorization into its subgroups $F = F \times \{e\}$ and $\Gamma = \{e\} \times \Gamma$; that is, $F \bowtie \Gamma = F \Gamma$ and $F\cap \Gamma = \{e\}$.

Conversely, every finite group $G$ endowed with an exact factorization $G = F \Gamma$ into its subgroups $F$ and $\Gamma$ gives rise to canonical actions by permutations $\Gamma \overset{\lhd}\longleftarrow
\Gamma \times F \overset{\rhd}\longrightarrow F$ making $(\Gamma, F)$ into a matched pair of groups.

\medbreak Suppose $H \cong k^\Gamma {}^\tau\#_\sigma kF$ is an abelian extension of $k^\Gamma$ by $kF$. It follows from \cite[Theorem 1.3]{gp-ttic}  that the category $\Rep H$ is Morita equivalent to the pointed fusion category $\C(F\bowtie\Gamma, \omega)$, where $\omega$ is a $3$-cocycle on $F\bowtie\Gamma$ arising from the pair $(\sigma, \tau)$ in an exact sequence due to G. I. Kac. In particular, there are equivalences of braided fusion categories $$\Z(\Rep H) \cong \Rep D(H) \cong \Rep D^\omega(F \bowtie \Gamma),$$ where $D^\omega(F \bowtie \Gamma)$ is the twisted Drinfeld double of $F\bowtie \Gamma$ \cite{dpr}. Note that $\Rep H$ is solvable if and only if the group $F \bowtie \Gamma$ is solvable.
\end{remark}

\begin{example}\label{ddoble} Let $G$ be a finite group. Then the Drinfeld
double $D(G)$ fits into a split cocentral abelian exact sequence
$$k \longrightarrow k^G \longrightarrow D(G) \longrightarrow kG
\longrightarrow k.$$ This exact sequence is associated to the adjoint action
$\lhd: G \times G \to G$, $h\lhd g = g^{-1}hg$, and to the trivial action $\rhd:
G \times G \to G$.
\end{example}

The following lemma describes the group of invertible objects of the category
$\Rep H$, when $H$ is an abelian extension.

\begin{lemma}\label{inv-abel} Suppose $H$ fits into an exact sequence
\eqref{abel-es}. Then there is an exact sequence
$$1\longrightarrow \widehat F \overset{\pi^*}\longrightarrow G(\Rep H)
\overset{i^*}\longrightarrow \Gamma_0 \longrightarrow 1,$$ where $\widehat F$
denotes the group of one-dimensional characters of $F$ and $\Gamma_0 = \{s\in
\Gamma^F:\, [\sigma_s] = 1 \textrm{ in } H^2(\Gamma, (k^F)^\times) \}$.
\end{lemma}

\begin{proof} The group $G(\Rep H)$ can be identified with the group $G(H^*)$ of
group-like elements in the dual Hopf algebra $H^*$.  In addition, $H^*$ fits
into an abelian extension
\begin{equation}\label{abel-dual}
k \longrightarrow k^F \overset{i^*}\longrightarrow H^*
\overset{\pi^*}\longrightarrow k\Gamma
\longrightarrow k.
\end{equation} The lemma follows from
\cite[Lemma 2.2]{ext-ty}. \end{proof}

\begin{remark}\label{inv-split} Keep the notation in Lemma \ref{inv-abel}. Note
that the dual exact sequence \eqref{abel-dual} is associated to the actions
$F \overset{\lhd'}\longleftarrow F \times \Gamma \overset{\rhd'}\longrightarrow
\Gamma$ defined in the form
$x \lhd' s = (s^{-1} \rhd x^{-1})^{-1}$ and $x \rhd' s = (s^{-1} \lhd
x^{-1})^{-1}$, for all $x\in F$, $s\in \Gamma$ \cite[Exercise 5.5]{ma-ext}.

Hence the exact sequence of groups of Lemma \ref{inv-abel} induces the transpose
of the action $\lhd'$ of $\Gamma_0$ on the abelian group $\widehat F$.

Clearly, \eqref{abel-dual} is split if and only if \eqref{abel-es} is split and,
if this is the case,
the exact sequence of groups in Lemma \ref{inv-abel} is split as well.

Therefore, in the case where $H$ is a split abelian extension, the group $G(\Rep H)$ is isomorphic to the semidirect product $\widehat F \rtimes \Gamma_0$ with respect to the action $\lhd'$ of $\Gamma_0$ on $\widehat F$.
\end{remark}

\begin{corollary}\label{inv-doble} Let $G$ be a finite group. Then the group of
invertible objects of $\Rep D(G)$ is isomorphic to the direct product $G/[G, G]
\times Z(G)$.
\end{corollary}

\begin{proof} This is a consequence of Lemma \ref{inv-abel}, in view of Example \ref{ddoble} and Remark \ref{inv-split}.
In fact, $\widehat G \cong G/[G, G]$ and the actions $\rhd ': G \times G \to G$ and $\lhd ': G \times G \to G$ in Remark \ref{inv-split} are given in this case by $h \rhd ' g=hgh^{-1}$ and $g \lhd ' h=g$, for all $g, h\in G$. Then $G_0=\{g\in G | h \rhd ' g=g, \forall h\in G \}=Z(G)$.
The Corollary follows from the fact that the action $\lhd '$ is the trivial one.
\end{proof}

\subsection{Examples associated to the symmetric group}
Let $n \geq 2$ be a natural number. The symmetric group $\Ss_n$ has an exact
factorization $\Ss_n = \langle z \rangle \Gamma$, where $\Gamma = \{\sigma\in
\Ss_n:\, \sigma(n) = n\} \cong \Ss_{n-1}$ and $z = (1 2 \dots n)$, so that
$\langle z \rangle \cong C_n$.
This exact factorization
induces mutual actions by permutations $\Ss_{n-1} \overset{\lhd}\longleftarrow
\Ss_{n-1} \times C_n \overset{\rhd}\longrightarrow C_n$
that make  $(\Ss_{n-1}, C_n)$ into a matched pair of groups.
The actions $\lhd, \rhd$ are determined by the relations
\begin{equation}\label{rel-mp} \sigma c   = (\sigma \rhd c) (\sigma \lhd c),
 \end{equation}
for all $\sigma \in \Gamma$, $c \in \langle z \rangle$.

Suppose $n$ is odd, so that $\langle z \rangle \subseteq \Aa_n$. Relations
\eqref{rel-mp} imply that the subgroup $\Gamma_+ = \Gamma \cap \Aa_n \cong
\Aa_{n-1}$ is stable under the action $\lhd$ of $\langle z \rangle$.
Therefore the actions $\lhd, \rhd$ induce by restriction a matched pair
$(\Aa_{n-1}, C_n)$.

\begin{remark}\label{even-notstbl} Let $\sigma \in \Gamma$. It follows from \eqref{rel-mp} that  $\sigma z = z^r (\sigma \lhd z)$, for some $0 \leq r \leq n-1$. Since $\sigma \lhd z \in \Gamma$ then $(\sigma \lhd z)(n) = n$, implying that $r = b(n)$, where $b = \sigma z$.

Suppose that $n\geq 4$ is even and $\sigma \in \Gamma \cap \Aa_n$. Since $z$ is an odd permutation and $\sigma \lhd z = z^{-r}\sigma z$, then $\sigma \lhd z$ is even if and only if $r$ is odd. Letting $\sigma = (12\dots (n-1))\in \Gamma \cap \Aa_n$, we find that $r = b(n) = \sigma z(n) = 2$; so that $\sigma \lhd z$ is an odd permutation. This shows that the subgroup $\Gamma_+ = \Gamma \cap \Aa_n \cong \Aa_{n-1}$ is not stable under the action $\lhd$ of $\langle z \rangle$ in this case.
\end{remark}

\medbreak Let us consider the associated Hopf algebras $J_n= k^{\Ss_{n-1}} \# kC_n$ and $K_n= k^{\Aa_{n-1}} \# kC_n$.

The categories $\Rep J_n$, $\Rep K_n$ are Morita equivalent to the categories
$\Rep \Ss_n$ and $\Rep\Aa_n$, respectively; see Remark \ref{2-morita}. In particular, $\Rep
J_n$ and $\Rep
K_n$ are not solvable, for all $n \geq 5$.

Observe that $J_n^*$ is a split abelian extension of $k^{C_n}$ by
$k\Ss_{n-1}$ associated to the actions $\lhd'$ and $\rhd'$ in Remark
\ref{inv-split}.

\begin{remark} Suppose $n \geq
5$.  It follows from \cite[Theorem 5.2]{qt-bicrossed}, that the Hopf
algebras $J_n$, $K_n$, $J_n^*$, $K_n^*$  admit no quasitriangular structure.
In particular, the fusion categories $\Rep J_n$, $\Rep K_n$, $\Rep J_n^*$ and $\Rep K_n^*$ admit no
braiding.
\end{remark}

In addition, there are equivalences of braided fusion categories
\begin{equation}
\Rep D(J_n) \cong \Rep D(\Ss_n), \quad \Rep D(K_n) \cong \Rep D(\Aa_n).
\end{equation}
It follows from Corollary \ref{inv-doble} that there are group isomorphisms
$G(D(\Ss_n)^*) \cong \Ss_n/[\Ss_n, \Ss_n] \times Z(\Ss_n) \cong \Zz_2$, for all
$n \geq 3$, and similarly,
$G(D(\Aa_n)^*) = 1$, for all $n \geq 5$.

\begin{lemma}\label{tann-sn} The pointed subcategory $\Rep D(\Ss_n)_{pt}\cong \Rep \Zz_2$ is a Tannakian subcategory of $\Rep D(\Ss_n)$.
\end{lemma}

\begin{proof} It follows from the description of the irreducible representations in \eqref{irrepns-bcpt}, that the one-dimensional representations of $D(\Ss_n)$ are parameterized by pairs $(s, U_s)$, where $s \in \Ss_n$ is a central element and $U_s$ is a one-dimensional representation of $\Ss_n$.
Since $Z(\Ss_n) = \{e\}$ and $\Ss_n$ has two non-isomorphic one dimensional representations trivial one and the sign representation $\Sg$, then $\Rep D(\Ss_n)_{pt}\cong \Rep \Zz_2$. Moreover, the  unique nontrivial element of $\Rep D(\Ss_n)_{pt}$ corresponds to the pair $(e, \Sg)$.
We have
$$s_{(e,\Sg),(e,\Sg)}= \frac{|\Ss_n|}{|\Ss_n|^2}\sum_{g\in \Ss_n} \Sg(e)\Sg(e) =\frac{|\Ss_n|}{|\Ss_n|}=1,$$
and
$$\theta_{(e,\Sg)}= \frac{\Sg(e)}{\deg \Sg}= 1,$$
 where $(s_{X, Y})_{X, Y \in \Irr(D(\Ss_n))}$ and $\theta$ denote the $S$-matrix and the ribbon structure of $\Rep D(\Ss_n)$, respectively.  See for instance \cite[Section 3.1]{NNW}.
 This shows that $\Rep D(\Ss_n)_{pt}$ is  a Tannakian subcategory, as claimed.
\end{proof}

\begin{lemma}\label{secjnkn} Let $n$ be an odd natural number. Then there is a central exact sequence of Hopf algebras
\begin{equation}\label{sec-jn}
k \longrightarrow k^{\Zz_2} \longrightarrow J_n \overset{\pi}\longrightarrow K_n
\longrightarrow k,
\end{equation}
where the map $\pi: J_n \to K_n$ is induced by the inclusion $\Aa_{n-1}
\subseteq \Ss_{n-1}$.
\end{lemma}

\begin{proof} The map $\pi$ is defined in the form $\pi = j \otimes \id: J_n = k^{\Ss_{n-1}} \# kC_n
\to K_n = k^{\Aa_{n-1}} \# kC_n$, where $j: k^{\Ss_{n-1}}
\to k^{\Aa_{n-1}}$ is the canonical Hopf algebra map. Then $\pi$ is a surjective Hopf algebra map.

Since the index of $K_n$ in  $J_n$ is $2$, then $J_n^{\co \pi} \cong k\Zz_2 \cong k^{\Zz_2}$ is a necessarily central Hopf subalgebra; see \cite[Corollary 1.4.3]{ssld}.
\end{proof}

\begin{proposition} \label{RepJn} \textit{(i)} $\Rep J_n$ is a $\Zz_2$-extension of $\Rep K_n$, for all odd natural number $n \geq 1$.

\textit{(ii)} $\Rep J_n$ is a $\Zz_2$-equivariantization of a fusion category $\D$, for all even natural number $n\geq 4$.
\end{proposition}

\begin{proof}
\textit{(i)}
This is an immediate consequence of Lemma \ref{secjnkn}. That is, since the sequence \eqref{sec-jn} is a central exact sequence of Hopf algebras, then $\Rep (J_n)$ is an $\Zz_2$-graded fusion category with trivial component $(\Rep (J_n))_0 = \Rep(K_n)$ \cite[Theorem 3.8]{gel-nik}. Therefore $\Rep J_n$ is a $\Zz_2$-extension of $\Rep K_n$.

\medbreak \textit{(ii)} Suppose that $n \geq 4$ is even. We first claim that $\Rep J_n$ is not a $\Zz_2$-extension of any fusion category.  To see this, first note that it follows from \cite[Lemma 3.4]{jm} that $G(J_n) = G(k^{\Ss_{n-1}}) \cong \Zz_2$, for all $n \geq 2$. Suppose that $K$ is a central Hopf subalgebra of $J_n$ such that $K \cong k^{\Zz_2}$, then $K$ must necessarily coincide with $kG(J_n)$.
 Observe that the action $\lhd: \Ss_{n-1} \times C_n \to \Ss_{n-1}$ gives rise to an action by algebra automorphism
 $$\rightharpoonup: C_n \times k^{\Ss_{n-1}} \to k^{\Ss_{n-1}} \ \ \text{such that} \ \ x \rightharpoonup e_\sigma = e_{\sigma \lhd x^{-1}},$$
 for all $x\in C_n$ and $\sigma \in \Ss_{n-1}$. If $\epsilon\neq\varphi\in G(J_n) = G(k^{\Ss_{n-1}})$, we have $\varphi(\sigma)=\Sg(\sigma)$, for all $\sigma\in\Ss_{n-1}$ and $\varphi = \sum_{\sigma\in\Ss_{n-1}}\Sg(\sigma)e_\sigma$.

Then  $G(J_n)\subseteq Z(J_n)$ if and only if $z\rightharpoonup\varphi=\varphi$, if and only if $\Aa_{n-1}$ is stable under the action $\lhd$ of $\langle z \rangle$. This contradicts the observation in Remark \ref{even-notstbl}. Hence $\Rep J_n$ is not a $\Zz_2$-extension of any fusion category, as claimed.

\medbreak Let $\E = \Rep D(\Ss_n)_{pt}$. By Lemma \ref{tann-sn}, $\E\cong \Rep \Zz_2$ is a Tannakian subcategory of $\Rep D(\Ss_n)$. Since $\Rep J_n$ is not a $\Zz_2$-extension of any fusion category, it follows from \cite[Propositions 2.9 and 2.10]{ENO2} that $\Rep J_n$ must be a $\Zz_2$-equiva\-riantization of a fusion category $\D$. We thus obtain (ii).
\end{proof}

\begin{lemma}\label{not-ext} Let $n \geq 5$ be an odd natural number and let $q$ be a prime number. Then
the following hold:

(i) The category $\Rep J_n$ is not a $\Zz_q$-equivariantization of any fusion
category.

(ii) Suppose that $\Rep J_n$ is a $\Zz_q$-extension of a fusion category $\tilde
\D$.
Then $q = 2$ and $\tilde \D = \Rep K_n$.

(iii) The category $\Rep K_n$ is neither a $\Zz_q$-equivariantization nor a
$\Zz_q$-extension of any fusion category. \end{lemma}

\begin{proof} Let $\C$ be one of the categories $\Rep J_n$ or $\Rep K_n$. If
$\C$ is a $\Zz_q$-extension of a fusion category $\tilde \D$, then the Drinfeld
center $\Z(\C)$ must contain a Tannakian subcategory equivalent to $\Rep \Zz_q$
which maps to the trivial fusion subcategory $\vect \subseteq \C$ under the
forgetful functor $U:\Z(\C) \to \C$. In this case, the category $\tilde \D$ is
canonically
determined by the corresponding Tannakian subcategory.
Dually, if $\C$ is a $\Zz_q$-equivariantization, then $\Z(\C)$ must contain a
Tannakian subcategory equivalent to $\Rep \Zz_q$
which embeds into $\C$ under the forgetful functor. See \cite[Propositions 2.9
and 2.10]{ENO2}.

Since $n \geq 5$, then the group of invertible objects of the Drinfeld center of
$\C$ coincides with $\Zz_2$ if $\C = \Rep J_n$, and it is trivial if $\C = \Rep
K_n$. Since $\Rep \Zz_q$ is a pointed fusion category, we get (iii).

Suppose that $\C = \Rep J_n$. Then $\C$ is a $\Zz_2$-extension of $\Rep K_n$.
This implies that the pointed subcategory of $\Z(\C)$ is a Tannakian subcategory
which maps to the trivial subcategory of $\C$ under the forgetful functor.
Hence, for every prime number $q$, $\C$ is not a $\Zz_q$-equivariantization of
any fusion category and we get (i).
On the other hand, if it is $\Zz_q$-extension, then $q = 2$ and the
corresponding Tannakian subcategory of $\Z(\C)$ coincides with the pointed
subcategory $\Z(\C)_{pt}$. Thus we obtain (ii).
This finishes the proof of the lemma.
\end{proof}

\begin{lemma}\label{inv-sdp} Let $p$ be an odd prime number. Then the group
$G(J_p^*)$ is a
semidirect product $\widehat{C_p} \rtimes \langle \sigma \rangle$ where $\sigma
\in \Ss_{p-1}$ is a $(p-1)$-cycle and the action of
$\langle \sigma \rangle$ on $\widehat{C_p}$ is induced by the action $\lhd': C_p
\times \Ss_{p-1} \to C_p$.
Moreover, the subgroup $G(K_p^*) \subseteq G(J_p^*)$ is the semidirect product
$\widehat{C_p} \rtimes \langle \sigma^2 \rangle$.
\end{lemma}

\begin{proof} The subgroup $\Ss_{p-1}^{C_p}$ of invariants of  $\Ss_{p-1}$ under
the action $\rhd'$ of $C_p$ coincides with the subgroup of invariants under the
action $\lhd$. It follows from \cite[Corollary 5.2]{jm} that $\Ss_{p-1}^{C_p}$
is cyclic generated by a $(p-1)$ cycle $\sigma$,
i.e. $\Ss_{p-1}^{C_p}=\langle \sigma \rangle$ and therefore $G(J_p^*)\cong \widehat{C_p} \rtimes \langle \sigma \rangle$.
It follows from this that the invariant subgroup $\Aa_{p-1}^{C_p}$ is also
cyclic generated by $\sigma^2$.
This implies the lemma, in view of Remark \ref{inv-split}.
\end{proof}

\begin{proposition}\label{triv-cent} Let $p$ be an odd prime number. Then the
Hopf algebras
$J_p$, $K_p$ satisfy the following properties:

(i) $\cd(J_p) = \cd (K_p) = \{1, p \}$.

(ii) The groups $G(J_p^*)$ and $G(K_p^*)$ have trivial centers.
\end{proposition}

\begin{proof} Part (i) follows from the description of irreducible
representations of crossed products in \cite{MW}.

\medbreak We next show (ii). Recall that $C_p = \langle z \rangle$, where $z = (12\dots p)$ and the actions
$\lhd$, $\rhd$ are determined by the relation $s c = (s\rhd c) (s \lhd c)$ in
$\Ss_p$. So that the actions $\lhd'$, $\rhd'$ are determined by $cs = (c \rhd'
s) (c\lhd' s)$ in
$\Ss_p$, for all $s \in \Ss_{p-1}$, $c \in C_p$.

It follows from the proof of \cite[Lemma 3.2]{jm} that $z \lhd' a_i = z^i$, for all $i = 1,
\dots, p-1$, where $a_i = (p-1, p-i)$. In addition, the stabilizer of $z$ under
the action $\lhd'$ coincides with the subgroup $F_z = \{a \in \Ss_{p-1}:\,
a(p-1) = p-1\} \cong \Ss_{p-2}$.
These imply that, for all $i = 1, \dots, p-1$, the stabilizer of $z^i$ coincides
with the subgroup
$F_{z^i} = \{a \in \Ss_{p-1}:\, a(p-i) = p-i\}$. In particular, $C_p$ has no nontrivial fixed points under the action $\lhd'$.

On the other hand, the nontrivial powers of the $(p-1)$-cycle $\sigma \in
\Ss_{p-1}$ have no fixed point in $\{1, \dots, p-1\}$. Hence no nontrivial power
of $\sigma$ acts trivially on $C_p$ under the action $\lhd'$. 

\medbreak By Lemma \ref{inv-sdp}, $G(J_p^*) = \widehat{C_p} \rtimes
\langle \sigma \rangle$ is a semidirect product with respect to the action
induced by $\lhd'$, where $\sigma$ is a $(p-1)$-cycle in $\Ss_{p-1}$.
Then the center of $G(J_p^*)$ consists of all pairs $(e, x)$, where $x \in
\langle \sigma \rangle$ acts trivially on $C_p$ under the action $\lhd'$.

 Similarly, $G(K_p^*) = \widehat{C_p} \rtimes
\langle \sigma^2 \rangle$ is a semidirect product with respect to the action
induced by $\lhd'$ and the center of $G(K_p^*)$ consists of all pairs $(e, x)$,
where $x \in
\langle \sigma^2 \rangle$ acts trivially on $C_p$ under the action $\lhd'$.

Thus we obtain that the centers of the groups $G(J_p^*)$ and $G(K_p^*)$ are both trivial.
\end{proof}

\begin{theorem}\label{jp-kp} Let $\tilde\C$ be a fusion category. Suppose that $\tilde \C$ is
Grothendieck equivalent to one of the categories $\Rep J_p$ or $\Rep K_p$. Then
$\tilde \C$ is not solvable.
\end{theorem}

\begin{proof} Suppose on the contrary that $\tilde \C$ is solvable and
Grothendieck equivalent to $\C$, where
$\C = \Rep J_p$ or $\Rep K_p$. It follows from Proposition \ref{groth-eq}, that $G(\tilde \C) \cong G(\C)$. By
Proposition \ref{triv-cent}, the groups of invertible objects of $\Rep J_p$ and
$\Rep K_p$ have trivial center. Then the center of $G(\tilde \C)$ is trivial as
well.
It follows from Lemma \ref{g-central} that, for every prime number $q$, the
category $\tilde \C$ is not a $\Zz_q$-equivariantization of any fusion category.
Therefore $\tilde \C$ must be a $\Zz_q$-extension of a fusion subcategory
$\tilde \D$, and $\tilde \D$ is also a solvable fusion category. Hence $\C$ is a
 $\Zz_q$-extension of a fusion subcategory $\D$ such that $\tilde \D$ is
Grothendieck equivalent to $\D$.
It follows from Proposition \ref{RepJn} and Lemma \ref{not-ext} that $\C = \Rep J_p$, $q = 2$ and $\D = \Rep
K_p$.
Applying the same argument to the solvable fusion category $\tilde \D$ we get a
contradiction. This shows that $\tilde \C$ cannot be solvable and finishes the
proof of the theorem.
 \end{proof}

\subsection{Fusion rules of $\Rep K_5$} In this subsection we determine
explicitly the fusion rules of the
category $\Rep K_p$ in the case $p = 5$. It follows from \cite{MW} that simple
objects of the category $\Rep K_5$ are parameterized by pairs $(s, \rho)$, where
$s$ runs over a set of representatives of the orbits of $C_5$ on $\Aa_4$ and
$\rho$ is an irreducible representation of the stabilizer $F_s \subseteq C_5$.
The dimension of the simple object $S_{s, \rho}$ corresponding to the pair $(s,
\rho)$ is given by the formula
$\dim S_{s, \rho} = [C_5: F_s]$.

The $C_5$-action on $\Ss_4$ is explicitly determined in \cite[Table 1]{jm}. We
have in this case that there are $10$ fixed points and the remaining 2 orbits
consist of $5$ distinct elements each.
Furthermore, there are exactly $4$ distinct fixed points $\sigma$ such that
$\sigma = \sigma^{-1}$ and both nontrivial orbits contain elements of order $2$.
In view of \cite[Theorem 4.8]{jm},
$\Rep K_5$ has $5$ invertible objects of order $2$ and the $5$-dimensional
simple objects are self-dual.

Let us denote by $Y, Y'$ the simple objects corresponding to the nontrivial
orbits $\mathcal O, \mathcal O'$, respectively. By \cite[Table 1]{jm}, we have
\begin{align*}& \mathcal O = \{(123), (243), (132), (13)(24), (234)\}, \\
 & \mathcal O' =
\{(124), (143), (134), (12)(34), (142)\}.\end{align*}
By Lemma \ref{inv-sdp}, the group $G(\Rep K_5)$ is isomorphic to the dihedral
group $D_5$ of order $10$.
The unique subgroup $R$ of order $5$ of $G(\Rep K_5)$ coincides therefore with
the
stabilizer of $Y$ and $Y'$ under left (or right) multiplication.
Since every element $s$ outside of $R$ is of order $2$, then $s\otimes Y \cong Y
\otimes s \cong Y'$.
So that we have a decomposition
\begin{equation}\label{a-b}
Y \otimes Y^* \cong Y\otimes Y \cong \bigoplus_{r \in R} r \oplus aY
\oplus bY',
\end{equation}
where $a, b \geq 0$ and $a+b = 4$.
Letting $F: \Rep K_5 \to \C(\Aa_{4}) = \Rep k^{\Aa_{4}}$ denote the
restriction functor, we obtain
\begin{align*}& F(Y) = V_{(123)} \oplus V_{(243)}\oplus V_{(132)}\oplus V_{
(13)(24)}\oplus V_{ (234)},\\
 & F(Y') = V_{(124)}  \oplus V_{(143)}  \oplus V_{(134)} \oplus V_{(12)(34)}
\oplus V_{(142)},
\end{align*}
where, for each $s \in \Aa_4$, $V_s$ denotes the one-dimensional simple
$k^{\Aa_4}$-module corresponding to $s$.
Comparing these relations with \eqref{a-b}, we find that $a = b = 2$. Hence the
fusion rules of $\Rep K_5$
are determined by the condition $G = G(\Rep K_5) \cong D_5$, $g \otimes Y = Y
\otimes g = Y'$, for every element or order $2$ of $G$, and
$$Y\otimes Y \cong \bigoplus_{r \in R} r \oplus 2Y
\oplus 2Y'\cong Y'\otimes Y',$$ where $R$ is the unique subgroup of order $5$ of
$G$.

\subsection{The dual Hopf algebras $J_n^*$, $K_n^*$}

Let $n \geq 2$ be a  natural number and let $H_n = J_n^*$.
Recall that there is a split exact sequence of Hopf algebras
\begin{equation}\label{hn-ln} k \longrightarrow k^{C_n} \longrightarrow H_n
\longrightarrow k\Ss_{n-1}
\longrightarrow k.
 \end{equation}
Suppose that $n$ is odd.  Let  $L_n = K_n^*$, so that there is a split exact sequence of Hopf algebras
\begin{equation}\label{hn-ln-2}k \longrightarrow k^{C_n} \longrightarrow L_n
\longrightarrow k\Aa_{n-1}
\longrightarrow k.
 \end{equation}
 Moreover, by Lemma \ref{secjnkn} there is a cocentral exact sequence
 \begin{equation}\label{cocentral}
      k \longrightarrow L_n \longrightarrow H_n \longrightarrow k{\Zz_2}
\longrightarrow k.
     \end{equation}

\begin{remark}\label{dual-next}  Suppose $n \geq 5$. Since $D(H_n) \cong D(J_n)$,  then $G(\Rep D(H_n)) \cong \Zz_2$. Let $q$ be a prime number.  If the category $\Rep H_n$ is a $\Zz_q$-extension or a $\Zz_q$-equivariantization of a fusion category, then $q = 2$.

Suppose $n$ is odd. In view of \cite[Proposition 3.5]{ext-ty}, $\Rep H_n$ is a
$\Zz_2$-equivariantiza\-tion of $\Rep L_n$. As in the proof of Lemma \ref{not-ext}, we obtain that if $n \geq 5$, the
category $\Rep H_n$ is not a $\Zz_q$-extension of any fusion category.
Similarly,  $\Rep L_n$ is not a $\Zz_q$-extension or a $\Zz_q$-equivariantization of any fusion category.
\end{remark}

\begin{lemma}\label{grupos-deg2} Let $n \geq 5$ be a natural number. Then the following hold.

(i) $G(\Rep H_n) \cong  \Zz_2$.

(ii) $\Rep (H_5)$ is of type $(1, 2;  2, 1;  3, 2;  4, 2; 8, 1)$ and $H_n$ has no irreducible representation of dimension 2, for all $n > 5$.

Assume in addition that $n$ is odd. Then

(iii) $G(\Rep L_n) = \uno$, if $n > 5$.

(iv) $\Rep L_5$ is of type $(1, 3; 3, 1; 4, 3)$ and $L_n$ has no irreducible representation of dimension 2.
 \end{lemma}

\begin{proof} Consider the exact sequences \eqref{hn-ln}, \eqref{hn-ln-2}. The respective
invariant subgroups $C_n^{\Ss_{n-1}}$ and $C_n^{\Aa_{n-1}}$ are both trivial.
Parts (i) and (iii) follow from Lemma \ref{inv-abel}.

\medbreak Since $H_n$ is a split abelian extension of $k^{C_n}$ by $k\Ss_{n-1}$ and the
action of $\Ss_{n-1}$ has two orbits $\{e\}$ and $\{z, \dots, z^{n-1}\}$, then
the simple $H_n$-modules are
classified by pairs $(t, \rho)$, where either $t = e$
and $\rho$ is an irreducible representation of $F_e = \Ss_{n-1}$, or
$t = z$ and $\rho$ is an irreducible representation of $F_z = \{a \in \Ss_{n-1}:\, a(n-1) = n-1 \} \cong  \Ss_{n-2}$. See \cite[Lemma 3.2]{jm}. If $S_{t, \rho }$ is the simple module corresponding to the pair $(t, \rho)$, we
have in addition  $\dim S_{e, \rho} = \dim \rho$, and $\dim S_{z, \rho} =
[\Aa_{n-1}: F_z]\dim \rho = (n-1)\dim \rho$. This implies the statement for $H_5$ in part (ii).

Suppose that $n > 5$.   Then $\dim S_{z,
\rho} > 2$, for all $\rho$.
As observed in Remark \ref{rmk-ansn-2}, the symmetric group $\Ss_{n-1}$ has no irredubible
representation of degree $2$. Therefore we also get that $\dim S_{e, \rho} \neq 2$, for all $\rho$. In conclusion $H_n$
has no irreducible representation of dimension 2, and we obtain (ii).

\medbreak Suppose that $n$ is odd. Similarly, $L_n$ is a split abelian extension of $k^{C_n}$ by $k\Aa_{n-1}$ and the
action of $\Aa_{n-1}$ has two orbits $\{e\}$ and $\{z, \dots, z^{n-1}\}$. Hence
the simple $L_n$-modules are
classified by pairs $(t, \rho)$, where either $t = e$
and $\rho$ is an irreducible representation of $F_e \cap \Aa_{n-1} = \Aa_{n-1}$, or
$t = z$ and $\rho$ is an irreducible representation of $F_z \cap \Aa_{n-1} \cong  \Aa_{n-2}$. This implies that $\Rep L_5$ is of the prescribed type.
As before, $\dim S_{z, \rho} > 2$, for all $\rho$, and since the alternating group $\Aa_{n-1}$ has no irredubible
representation of degree $2$, then also $\dim S_{e, \rho} \neq 2$, for all $\rho$. So that $L_n$
has no irreducible representation of dimension 2. This proves part (iv) and finishes the proof of the lemma.
\end{proof}

\begin{lemma}\label{nsol-type} Let $\C$ be a fusion category of type $(1, 3; 3, 1; 4, 3)$. Then $\C$ is not solvable.
\end{lemma}

\begin{proof} The assumption on the type of $\C$ implies that the simple objects of dimensions $1$ and $3$ generate a fusion subcategory $\B$ of $\C$ of type $(1, 3; 3, 1)$ and moreover, every fusion subcategory of $\C$ is contained in $\B$.

Suppose first that $\C$ is a $\Zz_q$-extension of a fusion subcategory $\C_0$, for some prime number $q$. Then necessarily $\C_0 = \B$ and $q = 5$. Hence we have a $\Zz_5$-faithful grading $\C = \C_0 \oplus \C_1 \oplus \dots \oplus \C_4$, with trivial component $\C_0 = \B$. But $\C$ has only 3 classes of simple objects outside of $\B$, entailing that such a decomposition is impossible.

Therefore $\C$ must be a $\Zz_q$-equivariantization of a fusion category $\D$, for some prime number $q$, where $\D$ is also a solvable fusion category. Thus $q = 3$ and $\C \cong \D^{\Zz_3}$. The description of simple objects of $\D^{\Zz_3}$ together with the assumption on the type of $\C$ imply that $\D$ must be of type $(1, 4; 4, 1)$; c.f. Formula \eqref{dim-equiv}.
Moreover,  the action (by group automorphisms) of $\Zz_3$ on the set of nontrivial invertible objects of $\D$ must be transitive, hence $G(\D) \cong \Zz_2 \times \Zz_2$.

On the other hand, letting $X$ be the unique noninvertible simple object of $\D$, we must have $$X^{\otimes 2} \cong \oplus_{Y \in G(\D)}Y \oplus 3X.$$ This means that $\D$ is a near-group fusion category of type $(G, 3)$, where $G =G(\D)$. Then it follows from \cite[Theorem 1.2]{siehler} that the group $G(\D)$ is cyclic, which leads to a contradiction. Therefore $\C$ cannot be solvable. This finishes the proof of the lemma.
\end{proof}

\begin{theorem}\label{dual-jn-kn} Let $\C$ be a fusion category and let $n \geq 5$ be a
natural number. Then we have:

(i) If $\C$ is Grothendieck equivalent to $\Rep H_n$,  then $\C$ is not
solvable.

(ii) Suppose that $n$ is odd. If $\C$ is Grothendieck equivalent to $\Rep L_n$, then $\C$ is not
solvable.
\end{theorem}

\begin{proof} To show part (i), suppose  first that $n > 5$. By Lemma \ref{grupos-deg2}, $G(\Rep H_n) \cong \Zz_2$ and $\Rep H_n$ has no simple objects of dimension $2$. The claim follows in this case from Proposition \ref{cor-cyc}.

Consider the case $n = 5$. Suppose on the contrary that $\C$ is Grothendieck equivalent
to $\Rep H_5$ and $\C$ is solvable. Then $\C$ is of type $(1, 2;  2, 1;  3, 2;  4, 2; 8, 1)$ and, for any prime $q$, $\C$ is not a $\Zz_q$-extension of any fusion subcategory, in view of Remark \ref{dual-next}. Therefore $\C$ must be a $\Zz_2$-equivariantization of a solvable fusion category $\D$ of dimension $60$. The description of the simple
objects of $\D^{\Zz_2}$ together with the assumption on the type of $\C$ imply that $\D$ must be of type $(1, 3; 3, 1; 4, 3)$. Lemma \ref{nsol-type} implies that $\D$ is not solvable, which is a contradiction. Thus we get (i).

\medbreak Let us show (ii). If $n = 5$, the result follows from Lemma \ref{nsol-type}. Suppose next that $n > 5$. Since a solvable fusion category contains nontrivial invertible objects, then part (ii) follows from Lemma \ref{grupos-deg2} (iii).
\end{proof}

\subsection{Further examples  associated to the symmetric group}

Let $n \geq 2$ be a natural number. Consider the matched pair $(\Aa_n, C_2)$,
where $C_2 = \langle (12) \rangle \subseteq \Ss_n$, the action $\lhd: \Aa_n
\times C_2 \to \Aa_n$ is given by conjugation in $\Ss_n$ and the action
$\rhd: \Aa_n \times C_2 \to C_2$ is trivial. The associated group $\Aa_n \bowtie
C_2$ is isomorphic to the symmetric group $\Ss_n$.

Let $B_n = k^{\Aa_n}\# kC_2$ be the split abelian extension associated to this
matched pair. We have $\cd(B_n) = \{1, 2\}$.
The fusion category $\Rep B_n$ is Morita equivalent to $\Ss_n$ and therefore
it is not solvable if $n \geq 5$.

\begin{remark} Since the action $\rhd$ is the trivial one and $C_2 \cong \Zz_2$, there is a cocentral exact sequence
\begin{equation}\label{cocentral-B_n}
      k \longrightarrow k^{\Aa_n} \longrightarrow B_n \longrightarrow k{\Zz_2}
\longrightarrow k.
     \end{equation}
In view of \cite[Proposition 3.5]{ext-ty} $\Rep B_n$ is a $\Zz_2$-equivariantization of $\Rep k^\Aa_n=\C(\Aa_n)$.

Moreover, since $\Rep B_n$ is Morita equivalent to $\Ss_n$ and the group of invertible objects of $\Z(\Rep \Ss_n)$ is cyclic of order 2, then
for all prime number $q$, $\Rep B_n$ is not a $\Zz_q$-extension of any fusion category and if it is a $\Zz_q$-equivariantization, then $q = 2$ (compare with Proposition \ref{RepJn}).
In particular, not being a $\Zz_q$-extention,  $B_n$ has no nontrivial central group-like elements; that is, $Z(B_n) \cap G(B_n) = \{1\}$.
\end{remark}

Our first statement concerns the dual Hopf algebra $B_n^*$. 

\begin{theorem}\label{bn*} Let $\C$ be a fusion category. Suppose that $\C$ is Grothendieck equivalent to $\Rep B_n^*$, $n \geq 5$. Then $\C$ is not solvable. \end{theorem}

\begin{proof} The dual Hopf algebra $B_n^*$ fits into a central exact sequence \begin{equation}\label{central-B_n*}
      k \longrightarrow k^{\Zz_2} \longrightarrow B_n \longrightarrow k{\Aa_n}
\longrightarrow k.
     \end{equation}
Therefore $\Rep B_n^*$ is a $\Zz_2$-extension of $\Rep \Aa_n$. Hence $\C$ is a $\Zz_2$-extension of a fusion category $\D$, which is Grothendieck equivalent to $\Rep \Aa_n$. 

Suppose on the contrary that $\C$ is solvable. Then so is $\D$ and therefore $\D_{pt} \neq \vect$. This implies that $\Aa_n$ has nontrivial one-dimensional representations, which is a contradiction. Then $\C$ cannot be solvable, as claimed.
\end{proof}

\begin{lemma}\label{g-repbn} The group $G(\Rep B_n)$ of invertible objects of
the category $\Rep B_n$ is isomorphic to the direct product
$\widehat{C_2} \times C_{\Aa_n}(12)$, where $C_{\Aa_n}(12)$ denotes the
centralizer in $\Aa_n$ of the transposition $(12)$.
\end{lemma}

\begin{proof} There is a group isomorphism $G(\Rep B_n) \cong G(B_n^*)$. On the other
hand, $B_n^*$ is a split abelian extension
$B_n^* \cong k^{C_2} \# k\Aa_n$, associated to the adjoint action
$\rhd': C_2 \times \Aa_n \to \Aa_n$ and the trivial action $\lhd': C_2 \times \Aa_n \to C_2$.
The result follows from Lemma
\ref{inv-abel}. \end{proof}

\begin{theorem}\label{bn} Suppose $n \geq 5$. Let $\tilde\C$ be a fusion category  Grothendieck equivalent to $\Rep B_n$. Then $\tilde\C$ is not solvable.
 \end{theorem}

\begin{proof} Suppose first that $n \geq 7$.
By Lemma \ref{g-repbn}, $G(\Rep B_n) \cong \widehat{C_2} \times C_{\Aa_n}(12)$.
Note that
$C_{\Aa_n}(12)$ contains the subgroup $\{\sigma \in \Aa_n:\, \sigma(1) = 1,
\sigma(2) = 2\} \cong \Aa_{n-2}$. Since $n \geq 7$, the group $\Aa_{n-2}$ is not
solvable. Then $G(\tilde\C)$ is not solvable neither and then $\tilde\C$ is not
solvable.

\medbreak It remains to consider the cases $n = 5$ and $6$.
It follows from Lemma \ref{g-repbn} that
$G(\Rep B_5) \cong \widehat {C_2} \times \Ss_3$ is non-abelian of order 12. Hence $\Rep B_5$ is of type $(1,12;2,27)$.
Similarly,  $G(\Rep B_6)  \cong \widehat {C_2} \times \Ss_4$ is non-abelian of order 48 and $\Rep B_6$ is of type $(1,48; 2,168)$.

Suppose that there exist a solvable fusion category $\tilde \C$ which is Grothendieck equivalent to $\C$, where $\C = \Rep B_5$ or $\Rep B_6$. By Proposition \ref{groth-eq}, $G(\tilde\C) \cong G(\C)$. Since, for every prime number $q$, $\C$ is not a $\Zz_q$-extension of any fusion category, we have that $\tilde\C$ must be a $\Zz_q$-equivariantization of a fusion subcategory $\D$, and $\D$ is also a solvable fusion category.
Moreover, $q=2$ because $\Zz_q \subseteq Z(G(\tilde\C))$ and $\FPdim \D= 60$ or  $\FPdim \D= 360$, respectively.
Then there is an exact sequence of fusion categories
$$\Rep\Zz_2 \to \tilde\C \to \D.$$
Since $\cd(\tilde\C)=\{1,2\}$, it follows that $\cd(\D)=\{1,2\}$. The previous exact sequence induces an exact sequence of groups
$$1 \to \widehat\Zz_2 \to G(\tilde\C) \to G_0(\D) \to 1,$$
where $\widehat\Zz_2$ denotes the group of invertible characters of $\Zz_2$ and $G_0(\D)$ is
the subgroup of $G(\D)$ consisting of isomorphism classes of invertible objects which are $\Zz_2$-equivariant. See \cite[Remark 3.1]{fusionrules-equiv}. As $\Zz_2$ is a cyclic group, we have that $G_0(\D)$ coincides with the subgroup of fixed points of the induced action of $\Zz_2$ on the group of invertible objects of $\D$.

Observe that, since $\C$ is also a $\Zz_2$-equivariantization of $\C_{\Zz_2} \cong \C(\Aa_5)$ or $\C(\Aa_6)$, respectively, then the group $G(\C) \cong G(\tilde \C)$ also fits into an exact sequence
$$1 \to \widehat\Zz_2 \to G(\C) \to G_0(\C_{\Zz_2}) \to 1.$$
In this case, the subgroup  $G_0(\C_{\Zz_2})$ is isomorphic to $C_{\Aa_5}(12)$ or $C_{\Aa_6}(12)$, respectively.
In addition, the group $G(\C)$ contains a unique normal subgroup of order 2. Therefore, $G_0(\D) \cong G_0(\C_{\Zz_2})$ is a non-abelian subgroup of $G(\D)$.

\medbreak Suppose first that $n = 5$. In this case $\C =\Rep B_5 \cong \C(\Aa_5)^{\Zz_2}$.
In particular, $G_0(\D)$ is a subgroup of order 6 of $G(\D)$. A counting argument shows that $G(\D)$ can be of type $(1,12;2,12)$ or else $\D$ is pointed.
Suppose that $\D$ is of type $(1,12;2,12)$. Then $\D_{pt}^{\Zz_2} \subseteq \tilde\C$ is a fusion subcategory of dimension $24$ and type $(1,12;2,3)$, containing the central subcategory $\Rep\Zz_2$. Therefore $\C$ has a fusion subcategory $\B$ of type $(1,12;2,3)$  containing the central subcategory $\Rep\Zz_2$. (In fact, as observed before, $G(\C)$ contains a unique normal subgroup of order $2$; see Lemma \ref{g-central}.)

Consider the de-equivariantization $\B_{\Zz_2} \subseteq \C(\Aa_5)$. We have that $\dim\B_{\Zz_2}=12$ and thus
$\B_{\Zz_2}=\C(H)$, where $H$ is a ($\Zz_2$-stable) subgroup of $\Aa_5$ of order 12. Since $G(\Rep B_5) \subseteq \B$, then the subgroup $H$ contains the invariant sugroup $\Aa_5^{\Zz_2} = C_{\Aa_5}(12) \cong \Ss_3$.
On the other hand, every subgroup of order $12$ of $\Aa_5$ is isomorphic to $\Aa_4$, then $H\cong \Aa_4$. This leads to a contradiction, because $\Aa_4$ has no subgroup of order $6$. This proves that $\D$ cannot be of type $(1,12;2,12)$.

Therefore $\D$ must be a pointed fusion category. In this case $\D = \C(\Gamma,\omega)$, where $\omega:\Gamma \times \Gamma \times \Gamma \to k^*$ is a 3-cocyle and $\Gamma$ is a solvable non-abelian subgroup of order 60. In addition  $\Zz_2$ acts on $\Gamma$ by group automorphisms  and $\Gamma^{\Zz_2} \cong \Ss_3$.
Since $\Gamma \neq \Aa_5$, $\Gamma$ can be isomorphic to $\Aa_4 \times \Zz_5$, $\Zz_{15} \rtimes \Zz_4$ or $\Zz_{15} \rtimes (\Zz_2 \times \Zz_2)$.

If $\Gamma \cong \Aa_4 \times \Zz_5$, the action of $\Zz_2$ must fix $\Aa_4$ and $\Zz_5$. Since $|\Gamma^{\Zz_2}|=6$ then $\Gamma^{\Zz_2} \subseteq \Aa_4$, and we reach a contradiction. Therefore $\Gamma \cong \Zz_{15} \rtimes \Zz_4$ or $\Zz_{15} \rtimes (\Zz_2 \times \Zz_2)$. In this case $\Gamma$ has a unique subgroup $L$ of order 15, and then $L$ is $\Zz_2$-stable and $\C(L)^{\Zz_2}$ is a fusion subcategory of $\tilde\C$ of dimension $30$. This implies that $\C$ has a fusion  subcategory of dimension 30. Such fusion subcategory must correspond to a quotient Hopf algebra of $B_5$ of dimension $30$, which is a contradiction, because $Z(B_5)\cap G(B_5) = \{1\}$. See \cite[Corollary 1.4.3]{ssld}. Thus  $\D$ cannot be pointed.
This proves that if $\tilde\C$ is Grothendieck equivalent to $\Rep B_5$ then $\tilde\C$ is not solvable.

\medbreak Finally, let us consider the case $n = 6$. In this case we have $\C:=\Rep B_6 \cong \C(\Aa_6)^{\Zz_2}$.
On the other hand, $G_0(\D)$ is a subgroup of order 24 of $G(\D)$. As before, one can see that $\D$ must be of type $(1,24;2,84)$, $(1,72;2,72)$, $(1,120;2,60)$, or else $\D$ is pointed.

Suppose that $\D$ is of type $(1,72;2,72)$ or $(1,120;2,60)$. In these cases $\D_{pt}^{\Zz_2} \subseteq \tilde\C$ is a fusion subcategory of dimension  144 or 240, respectively, containing the central subcategory $\Rep\Zz_2$. Therefore $\C$ has a fusion subcategory $\B$ of dimension 144 or 240, respectively, containing the central subcategory $\Rep\Zz_2$. The de-equivariantization $\B_{\Zz_2} \subseteq \C(\Aa_6)$ is of dimension $\dim\B_{\Zz_2}=72$ or 120, respectively. Then $\B_{\Zz_2}=\C(H)$, where $H$ is a $\Zz_2$-stable subgroup of $\Aa_6$ of order $72$ or 120, respectively. Since $\Aa_6$ has no subgroups of order 72 or 120, it follows that these types are not possible for $\D$.

Suppose next that $\D$ is of type $(1,24;2,84)$.  It follows from the description of the simple objects of $\D^G$ and the fact that $\cd(\D) = \cd(\C) =\{1,2\}$, that $\Zz_2$ acts trivially on the set $\Irr(\D)$; see \eqref{dim-equiv}. In particular, $G(\D) = G_0(\D) \cong C_{\Aa_6}(12) \cong \Ss_4$.

Since $\D$ is  solvable, then it is a $\Zz_p$-extension or a $\Zz_p$-equivariantization, where $p$ is a prime number that divides the dimension of $\D$, which is $360$. If $\D$ were a $\Zz_p$-equivariantization then, by Lemma \ref{g-central}, $\Zz_p \subseteq Z(G(\D))$, which is a contradiction because $G(\D) \cong \Ss_4$. Therefore $\D$ must be a $\Zz_p$-extension of a fusion subcategory $\D_e$.
The fusion subcategory $\D_e$ is  of dimension 72, 120 or 180. Furthermore, $\D_e$ must be stable under the action of  $\Zz_2$, since this action is trivial on $\Irr(\D)$. As before, this implies that $\C$ contains a fusion subcategory $\B$ of dimension 144, 240 or 360, respectively, containing the central subcategory $\Rep\Zz_2$. Hence $\B_{\Zz_2}=\C(H)$, where $H$ is a $\Zz_2$-stable subgroup of $\Aa_6$ with order 72, 120 or 180, respectively. But $\Aa_6$ has no subgroups neither of order 72, 120 nor 180, therefore the type $(1,24;2,84)$ is also impossible for $\D$.

\medbreak Suppose finally that $\D$ is a solvable pointed fusion category. We have $\D = \C(\Gamma,\omega)$, where $\omega:\Gamma \times \Gamma \times \Gamma \to k^*$ is a 3-cocyle and $\Gamma$ is a solvable non-abelian subgroup of order $360$. In addition  $\Zz_2$ acts on $\Gamma$ by group automorphisms and the subgroup $\Gamma_0$ of fixed points of $\Gamma$ under this $\Zz_2$-action is of order  $24$. Let $S$ be a Sylow 5-subgroup of $\Gamma$. Since $\Gamma$ is  solvable, there exist $H$, a Hall $\{2,5\}$-subgroup of $\Gamma$, and $K$, a Hall $\{3,5\}$-subgroup of $\Gamma$, such that $S \subseteq H$ and $S \subseteq K$. A counting argument shows that $S\unlhd H$ and $S\unlhd K$ and so $S\unlhd \langle H,K \rangle = \Gamma $. Hence $S$ is the unique Sylow 5-subgroup of $\Gamma$ and then $S$ is $\Zz_2$-stable. In this case $\C(S,\omega|_S)^{\Zz_2}$ is a fusion subcategory of $\tilde\C$ of  dimension $10$, containing the central subcategory $\Rep\Zz_2$. Therefore $\C$ has a fusion subcategory $\B$ with dimension 10, containing the central subcategory $\Rep\Zz_2$. The de-equivariantization $\B_{\Zz_2} \subseteq \C(\Aa_6)$ is of dimension  $5$. Then $\B_{\Zz_2}=\C(T)$, where $T$ is a $\Zz_2$-stable subgroup of $\Aa_6$ of order $5$. We have that $T=\{\id,(abcde),(acebd),(adbec),(aedcb)\}$, and without loss of generality we may assume $a=1$ and $b=2$. We thus reach a contradiction, since $(12)(12cde)(12)=(21cde) \neq (1ce2d)$. This proves that $\widetilde\C$ cannot be solvable and finishes the proof of the theorem.
 \end{proof}

\section{Solvability and fusion rules of a braided fusion category}\label{solv-fr-bfc}

Let $\C$ be a fusion category. Suppose that $\FPdim \C$ is an integer (which is
always the case if $\C$ is solvable). Then the adjoint subcategory $\C_{ad}$ is
integral \cite[Proposition 8.27]{ENO}.

Assume that $\C$ is braided. Recall that $\C$ is solvable and integral, then
either it is pointed or it contains a nontrivial Tannakian subcategory
\cite[Proposition 5.2]{witt-wgt}.

\begin{theorem}\label{e-tann} Let $\C$, $\tilde\C$  be Grothendieck
equivalent braided fusion categories. Suppose that $\C$ is solvable. Then the
following hold:

(i) $\tilde \C_{pt}$ is a solvable fusion category and it is not trivial if
$\tilde \C$ is not trivial.

(ii) If $\tilde \C$ is not pointed, then it contains a nontrivial Tannakian
subcategory.
 \end{theorem}

\begin{proof} Since $\C$ is solvable, \cite[Proposition 4.5 (iv)]{ENO2} implies
that $\C_{pt}
\neq \vect$. In addition $\C_{pt}$ is a solvable fusion category. Hence $\tilde
\C_{pt} \neq \vect$. We have $\C_{pt} \cong \C(G(\C),
\omega)$ for some invertible 3-cocycle $\omega$ on $G(\C)$. By assumption $\C$
is solvable, hence the group $G(\C)$ is solvable.

Moreover, $\tilde \C_{pt}$ is Grothendieck equivalent to
$\C_{pt}$ and therefore there exists an isomorphism of groups $G(\C) \cong
G(\tilde \C)$. Hence $G(\tilde
\C)$, and \textit{a fortiori} also $\tilde \C_{pt}$, are solvable.
This shows part (i).

\medbreak Suppose that $\tilde \C$ is not pointed, so that $\C$ is not pointed
neither. 
Note that $\C_{ad}$ is Grothendieck equivalent to $\tilde \C_{ad}$. If $\C_{ad}$
is a proper fusion subcategory, then an inductive argument implies that $\tilde
\C_{ad}$ is solvable and therefore so is $\tilde \C$, because it is a $U(\tilde
\C)$-extension of
$\tilde \C_{ad}$ and the universal grading group $U(\tilde \C)$ is abelian.
Hence we may assume that
$\tilde \C_{ad} = \tilde \C$ (in particular, the same is true for $\C$). Since
$\C$ is solvable, then its Frobenius-Perron dimension is an integer and
therefore $\C$ is in fact integral. Then $\tilde \C$ is also integral. To show part (ii) we may assume that $\tilde \C$ is
not solvable, in view of \cite[Proposition 5.2]{witt-wgt}.

\medbreak By part (i), $\tilde \C_{pt}$ is solvable and not trivial.
Note that $\tilde \C$ cannot contain any nontrivial non-degenerate fusion
subcategory. In fact, if $\C$ were non-degenerate, then $\tilde\C_{ad} = \tilde
\C_{pt}' \subsetneq \tilde \C$, against the assumption. If, on the other hand,
$\tilde \D \subseteq \tilde \C$ were a proper non-degenerate subcategory, then
$\tilde \C \cong \tilde \D \boxtimes \tilde \D'$, and both $\tilde \D$ and
$\tilde \D'$ are Grothendieck equivalent to fusion subcategories of $\C$. An
inductive argument implies that $\tilde \D$ and $\tilde \D'$ are solvable and
therefore so is $\tilde \C$.

Suppose that $\tilde \C$ contains no nontrivial Tannakian subcategory. It
follows from \cite[Lemma 7.1]{witt-wgt} that $\tilde \C_{pt} = \tilde \C' \cong
\svect$ and $G[\tilde X] = \uno$, for all simple object $\tilde X$ of $\tilde
\C$.
This implies that $\FPdim \C_{pt} = 2$ and $G[X] = \uno$, for all simple object
$X$ of $\C$.

On the other hand, since $\C$ is solvable and $\C_{ad} = \C$, then $\C$ must be
a
$\Zz_p$-equivariantization of a fusion category $\D$ for some prime number $p$.
In particular $\C$ contains a (pointed) fusion subcategory of dimension  $2$, and therefore $p = 2$.
It follows from Lemma \ref{simple-p} that $\C$ has a simple object $X$ of Frobenius-Perron dimension $2$.
In addition, for every such simple object $X$, we have $G[X] = \uno$.

The Nichols-Richmond theorem implies that $\C$ contains a fusion subcategory
$\overline\C$ of type $(1, 2; 2, 1; 3, 2)$ or $(1, 3; 3, 1)$ or $(1, 1; 3, 2; 4,
1; 5, 1)$; see \cite[Theorem 11]{NR}, \cite[Theorem 3.4]{fusion-lowdim}. The first possibility cannot
hold in this case, because the unique simple object of dimension $2$ of
$\overline\C$ is necessarily stable under the action of $G(\overline\C) \cong
\Zz_2$. The second possibility contradicts the assumption that $\FPdim \C_{pt} =
2$. The third possibility is also discarded because $\overline\C$ must be a
solvable fusion category, whence $\overline\C_{pt} \neq \vect$.
This contradiction shows that $\tilde\C$ must contain a Tannakian subcategory,
and hence (ii) holds.
\end{proof}

\begin{proposition}\label{e-prime} Let $\C$ be a braided fusion category.
Suppose that $\E
\subseteq \C$ is a Tannakian subcategory. Then  $\C$ is solvable if and only if
$\E'$ is solvable.
\end{proposition}

\begin{proof} If $\C$ is solvable, then every fusion subcategory of $\C$ is
solvable. In particular, $\E'$ is solvable, showing the 'only if' direction.
Conversely, suppose that $\E'$ is solvable. Since $\E$ is a Tannakian
subcategory, it is symmetric, and therefore $\E \subseteq \E'$. Then $\E$ is
solvable. Let $G$ be a finite group such that $\E \cong \Rep G$ as braided
fusion categories. Then the group $G$ is solvable, by \cite[Proposition 4.5
(ii)]{ENO2}.

Consider the $G$-crossed braided fusion category $\C_G$, so that $\C \cong
(\C_G)^G$ is an equivariantization. Furthermore, the category $\C_G$ is a
$G$-graded fusion category, and the neutral component $\C_G^0$ of this grading
satisfies $(\C_G^0)^G \cong \E'$ \cite[Proposition 4.56 (i)]{DGNOI}. Therefore
$\C_G^0$ is solvable. Since $G$ is solvable, then so is $\C_G$ and also $\C
\cong (\C_G)^G$. This proves the 'if' direction and finishes the proof of the
proposition.
\end{proof}

 \begin{remark} Let $\tilde \C$ be a braided fusion category. Suppose that $\tilde \C$ is Grothendieck equivalent to a solvable braided fusion category $\C$ and $\tilde \C$ is not solvable.
 Assume in addition that $\FPdim \tilde \C$ is minimum with respect to these properties.
 
Then $\tilde \C$ must satisfy the following conditions:

(i) $\tilde \C_{ad} = \tilde \C$.

(ii) $\tilde \C_{pt} \neq \vect$ is a solvable fusion subcategory  and $(\tilde
\C_{pt})' = \tilde \C$.

(iii) $\tilde \C$ contains a nontrivial Tannakian subcategory and for every
Tannakian subcategory $\tilde \E$, $\tilde \E' = \tilde \C$.

(iv) $\tilde \C$ contains no proper non-degenerate fusion subcategories.

\medbreak Indeed, (i) and (iv) can be shown as in the proof Theorem \ref{e-tann},   (ii) follows from (i) and Lemma \ref{cent-cpt},  and (iii) follow from Theorem \ref{e-tann} and Proposition \ref{e-prime}.
\end{remark}

\section{The character table of a spherical fusion category}\label{s-char-tbl}

\subsection{Spherical fusion categories}

A \emph{spherical structure} on a fusion category $\C$ is a natural  isomorphism of tensor functors $\psi: \id_\C \to (\; )^{**}$ such that $$d_+(X) = d_-(X),$$ for all objects $X$ of $\C$, where $d_\pm(X) = \Tr_\pm(\id_X)$, and for every endomorphism $f: X \to X$, $\Tr_\pm(f) \in k$ are defined as the compositions
$$\Tr_+(f): \uno \longrightarrow X \otimes X^* \overset{\psi_Xf \otimes \id}\longrightarrow X^{**} \otimes X \longrightarrow \uno,$$
$$\Tr_-(f): \uno \longrightarrow X^* \otimes X^{**} \overset{\id\otimes f\psi_X^{-1} }\longrightarrow X^{*} \otimes X \longrightarrow \uno.$$
Let $\C$ be a spherical fusion category, that is, a fusion category endowed with a spherical structure. The quantum dimension of $X \in \C$ is denoted by $d_X : = d_+(X) = d_-(X)$, and the quantum dimension of $\C$ is defined in the form $\dim \C = \sum_{X \in \Irr(\C)} d_X^2$. The \emph{quantum trace} of an endomorphism $f: X \to X$ is denoted by $\Tr(f)  =\Tr_+(f) = \Tr_-(f)$. See \cite[Subsection 2.4.3]{DGNOI}, \cite[Subsection 2.2]{ENO}.

\medbreak Recall that a fusion category is called \emph{pseudo-unitary} if its global dimension coincides with its Frobenius-Perron dimension. By \cite[Proposition 8.24]{ENO}, every weakly integral fusion category is pseudo-unitary. It is shown in \cite[Proposition 8.23]{ENO} that every pseudo-unitary fusion category admits a canonical spherical structure with respect to which quantum dimensions of objects coincide with their Frobenius-Perron dimensions.

\subsection{Modular categories and $S$-matrices}

A \emph{premodular} category is a braided fusion category equipped with a spherical structure. Equivalently, $\C$ is a braided fusion category endowed with a \emph{ribbon structure}, that is, a natural automorphism $\theta: \id_\C \to \id_\C$ satisfying
\begin{equation}\label{bal}\theta_{X \otimes Y} = (\theta_X \otimes \theta_Y) c_{Y, X}c_{X, Y}, \quad \theta_X^* = \theta_{X^*},\end{equation}
for all objects $X, Y$ of $\C$ \cite{bruguieres}, \cite[Subsection 2.8.2]{DGNOI}.

Let $\C$ be a premodular category. The \emph{central charge} of $\C$ is the ratio
$$\xi(\C) = \frac{\tau^+(\C)}{\sqrt{\dim \C}},$$ where $\sqrt{\dim \C}$ is the positive square root and $\tau^+(\C) = \sum_{X \in \Irr(\C)}\theta_Xd_X^2$. See \cite[Subsection 6.2]{DGNOI}.

\medbreak The $S$-matrix of $\C$ is defined in the form $S = (S_{XY})_{X, Y \in \Irr(\C)}$, where for all $X, Y \in \Irr(\C)$, $$S_{X, Y} = \Tr(c_{Y, X}c_{X, Y}) \in k$$ is the quantum trace of the squared braiding $c_{Y, X}c_{X, Y}:X \otimes Y \to Y \otimes X$.

\medbreak A premodular  category $\C$ is called \emph{modular} if the $S$-matrix is non-degenerate \cite{turaev-b} or, equivalently, if it is non-degenerate \cite[Proposition 3.7]{DGNOI}.

If $\C$ is a spherical fusion category, then its Drinfeld center $\Z(\C)$ is a modular category of global dimension $\dim \Z(\C) = (\dim \C)^2$ and central charge $\xi(\Z(\C)) = 1$ \cite{mueger-ii}, \cite[Example 6.9]{DGNOI}.

Suppose that $\C$ is a modular category.  Then for every $X, Y, Z \in \Irr(\C)$, the multiplicity $N_{XY}^Z$ of $Z$ in the tensor product $X \otimes Y$ is given by the \emph{Verlinde formula}:
\begin{equation}\label{verlinde}N_{XY}^Z = \frac{1}{\dim \C}\sum_{T \in \Irr(\C)} \frac{S_{XT} \, S_{YT} \, S_{Z^*T}}{d_T},\end{equation}
where $d_T$ denotes the quantum dimension of the object $T$ and $\dim \C$ is the quantum dimension of $\C$. See \cite[Theorem 3.1.14]{BK}.

\subsection{$S$-equivalence of spherical fusion categories}\label{s-equiv}

\begin{definition} Let $\C$ and $\D$ be spherical fusion categories. We shall say that $\C$ and $\D$ are \emph{$S$-equivalent} if there exists a bijection $f: \Irr(\Z(\C)) \to \Irr(\Z(\D))$ such that $f(\uno) = \uno$ and $S_{f(X), f(Y)} = S_{X, Y}$, for all $X, Y \in \Irr(\C)$.
\end{definition}

The following lemma summarizes some of the main properties of $S$-equivalence.

\begin{lemma}\label{s-eq} Let $\C$ and $\D$ be spherical fusion categories and suppose that $f:\Irr(\Z(\C))\to \Irr(\Z(\D))$ is an $S$-equivalence. Then the following hold:
\begin{enumerate}\item[(i)] $d_{f(X)} = d_X$, for all $X \in \Irr(\C)$. In particular, $\dim \Z(\C) = \dim \Z(\D)$.
\item[(ii)] $f:\Irr(\Z(\C))\to \Irr(\Z(\D))$ is a Grothendieck equivalence.
\item[(iii)] For every fusion subcategory $\E$ of $\Z(\C)$ we have $f(\E') = f(\E)'$. In particular, $f(\E)$ is symmetric (respectively, non-degenerate) if and only if so is $\E$.
\item[(iv)] For every fusion subcategory $\E$ of $\Z(\C)$, $f$ maps the projective centralizer of $\E$ to the projective centralizer of $f(\E)$.
\end{enumerate}
\end{lemma}

\begin{proof}  For every simple object $X$ of $\Z(\C)$ we have $d_X = d_\uno d_X =  S_{\uno, X} = S_{\uno, f(X)} = d_\uno d_{f(X)} = d_{f(X)}$, and we get (i).  Now part (ii) follows from (i) and the Verlinde formula \eqref{verlinde}.
Part (iii) follows from the fact that two simple objects $X$ and $Y$ centralize each other if and only if $S_{X, Y} = d_Xd_Y$.

We now show part (iv). Let $X$ and $Y$ be simple objects of $\C$. It follows from \cite[Proposition 3.22]{DGNOI} that $X$ belongs to the projective centralizer of $Y$ if and only if $X$ belongs to the centralizer of $Y \otimes Y^*$. In view of part (iii) this happens if and only if $f(X)$ centralizes $f(Z)$, for all $Z \in \Irr(\C)$ such that $N^Z_{Y\otimes Y^*} \neq 0$. Since, by (ii), $f$ is a Grothendieck equivalence, then $f(Y)^* = f(Y^*)$ (Proposition \ref{groth-eq} (iii)), and it follows that the last condition is equivalent to the condition that $f(X)$ centralizes $f(Y)\otimes f(Y)^*$, that is, $f(X)$ belongs to the projective centralizer of $f(Y)$.
\end{proof}

\begin{theorem}\label{s-gpttic} Let $\C$ and $\D$ be $S$-equivalent spherical fusion categories. Then $\C$ is group-theoretical if and only if so is $\D$.
\end{theorem}

\begin{proof} We have that $\C$ is group-theoretical if and only if $\Z(\C)$ is group-theoretical. Suppose that this is the case. In particular $\Z(\C)$, and therefore also $\Z(\D)$, are integral. Since $\Z(\C)$ is a modular category, \cite[Corollary 4.14]{dgno-gpttic} implies that it contains a symmetric subcategory $\E$ such that $\E'_{ad} \subseteq \E$. Since every $S$-equivalence preserves centralizers, symmetric subategories and is a Grothendieck equivalence between $\Z(\C)$ and $\Z(\D)$, this implies that $f(\E)$ is a symmetric subcategory of $\Z(\D)$ and $f(\E)'_{ad} = f(\E'_{ad}) \subseteq f(\E)$ (see Proposition \ref{groth-eq} (iv)). Hence $\Z(\D)$ and therefore also $\D$ are group-theoretical. This implies the theorem.
\end{proof}

\begin{lemma}\label{s-eq-groups} Let $G$ and $\Gamma$ be finite groups and let $\omega: G \times G \times G \to k^*$, $\omega': \Gamma \times \Gamma \times \Gamma \to k^*$ be $3$-cocycles on $G$ and $\Gamma$, respectively. Suppose that the categories $\C(G, \omega)$ and $\C(\Gamma, \omega')$ are $S$-equivalent. Then $G$ is solvable if and only if so is $\Gamma$.
\end{lemma}

\begin{proof} It is enough to show the 'if' direction. Thus, let us assume that $G$ is solvable. Let $f: \Irr(\Z(\C(\Gamma, \omega'))) \to \Irr(\Z(\C(G, \omega)))$ be an $S$-equivalence. The center of the category $\C(\Gamma, \omega')$ contains a Tannakian subcategory $\E$ equivalent to $\Rep \Gamma$ as braided fusion categories. In view of Lemma \ref{s-eq}, $f(\E)$ is a symmetric fusion subcategory of $\Z(\C(G, \omega))$ which is Grothendieck equivalent to $\Rep \Gamma$. Being symmetric, $f(\E)$ is equivalent as a fusion category to the category $\Rep F$ for some finite group $F$. Then $F$ is solvable because $\Z(\C(G, \omega))$ is solvable, by \cite[Proposition 4.5]{ENO2}.
Since the categories $\Rep \Gamma$ and $\Rep F$ are Grothendieck equivalent, then the groups $\Gamma$ and $F$ have the same character table. This implies that $\Gamma$ is solvable. Hence $\C(\Gamma, \omega')$ is solvable, as claimed. \end{proof}

\begin{theorem}\label{s-equiv-gt} Let $\C$ and $\D$ be $S$-equivalent spherical fusion categories and suppose that $\C$ is group-theoretical. Then $\C$ is solvable if and only if $\D$ is solvable.
\end{theorem}

\begin{proof} Since $\C$ is group-theoretical, $\Z(\C)$ is equivalent to the center of a pointed fusion category $\Z(\C(G, \omega))$, for some finite group $G$ and $3$-cocycle $\omega$ on $G$. Hence $\Z(\C)$ contains a Tannakian subcategory $\E$ equivalent to $\Rep G$ as braided fusion categories, such that $(\dim \E)^2 = \dim \Z(\C)$.

 Being Grothendieck equivalent to $\Z(\C)$, $\Z(\D)$ is also group-theoretical, in view of Theorem \ref{s-gpttic}. Thus $\Z(\C)$ is an integral modular category of dimension $(\dim \D)^2$ and central charge $1$. Note in addition that if $f$ is an $S$-equivalence, then $f(\E)$ is a symmetric subcategory of $\Z(\D)$ such that $\dim \Z(\D) = (\dim f(\E))^2$. Theorem 4.8 of \cite{dgno-gpttic} implies that $\Z(\D)$ is equivalent to the center of a pointed fusion category, that is, $\Z(\D) \cong \Z(\C(\Gamma, \omega'))$, for some finite group $\Gamma$ and $3$-cocycle $\omega'$ on $\Gamma$. Then the  theorem follows from Lemma \ref{s-eq-groups}. \end{proof}

\bibliographystyle{amsalpha}

\end{document}